\documentclass[12pt]{amsart}
\usepackage{amssymb}
\usepackage{mathrsfs}
\usepackage{graphicx}
\renewcommand\a{\alpha}
\renewcommand\b{\beta}

\renewcommand\d{\delta}
\newcommand\la{\lambda}

\newcommand\vf{\varphi}

\renewcommand\t{\tau}

\newcommand\Om{\Omega}

\newcommand\vL{\varLambda}

\newcommand\vG{\varGamma}
\newcommand\ve{\varepsilon}

\newcommand{\ZZ}{\mathbb Z}

\newcommand\Bp{\mathbf p}

\newcommand\Bla{\boldsymbol\lambda}
\newcommand\Bmu{\boldsymbol\mu}
\newcommand\Bnu{\boldsymbol\nu}

\newcommand\Bt{\boldsymbol\t}

\newcommand\CB{\mathcal{B}}
\newcommand\ZC{\mathcal{C}}

\newcommand\CI{\mathcal{I}}

\newcommand\CT{ \mathcal{T}}

\newcommand\FS{\mathfrak S}
\newcommand\Fu{\mathfrak u}
\newcommand\Fv{\mathfrak v}

\newcommand\Fs{\mathfrak s}

\newcommand\Ft{\mathfrak t}

\newcommand\wh{\widehat}
\newcommand\wt{\widetilde}

\newcommand\ol{\overline}

\newcommand\Ra{\Rightarrow }

\newcommand{\lan}{\langle}
\newcommand{\ran}{\rangle}

\newcommand{\ra}{\rightarrow }
\newcommand{\ot}{\otimes}

\newcommand{\sh}{^{\sharp} }

\newcommand\Hom{\operatorname{Hom}}
\newcommand\End{\operatorname{End}}

\newcommand\shp{^{\sharp}}

\newcommand{\rad}{\operatorname{rad}}

\newcommand{\isom}{\,\raise2pt\hbox{$\underrightarrow{\sim}$}\,}
\newcounter{ichi}
\newcommand{\roi}{\roman{ichi}}
\newcounter{ni}
\newcommand{\roii}{\roman{ni}}
\newcounter{san}
\newcommand{\roiii}{\roman{san}}
\newcounter{yon}
\newcommand{\roiv}{\roman{yon}}
\newcounter{go}
\newcommand{\rov}{\roman{go}}
\newcounter{roku}
\newcounter{nana}
\newcommand{\Sc}{\mathscr{S}}
\newcommand{\Sp}{\mathscr{S}^\Bp}
\newcommand{\oSp}{\ol{\Sc}^\Bp}

\newcommand{\tSp}{\wt{\Sc}^{\Bp}}
\newcommand{\He}{\mathscr{H}}
\newcommand{\A}{\mathscr{A}}
\newcommand{\aA}{\A^\a}
\newcommand{\oA}{\ol{\A}^\a}
\newcommand{\tA}{\wt{\A}^\a}
\newcommand{\eqA}{\stackrel{\,\,_{\A}}{\sim}}
\newcommand{\eqaA}{\stackrel{\,_\a}{\sim}}
\newcommand{\eqtA}{\stackrel{\,_{\wt{\a}}}{\sim}}
\newcommand{\eqoA}{\stackrel{\,_{\ol{\a}}}{\sim}}
\newcommand{\eqAl}{\stackrel{\,\,_{\sh\A}}{\sim}}

\newcommand{\eqtAl}{\stackrel{\,_{\sh\wt{\a}}}{\sim}}

\newcommand{\eqtAt}{\approx}

\newtheorem{thm}{Theorem}[section]
\newtheorem{lem}[thm]{Lemma}
\newtheorem{cor}[thm]{Corollary}
\newtheorem{prop}[thm]{Proposition}

\newtheorem{ass}[thm]{Assumption}

\def \para{\refstepcounter{thm} \par\medskip\noindent
                \textbf{\thethm .} }

\def \remark{\refstepcounter{thm} \par\medskip\noindent
                \textbf{Remark \thethm .} }

\def \remarks{\refstepcounter{thm} \par\medskip\noindent
                \textbf{Remarks \thethm .} }

\def \exam{\refstepcounter{thm} \par\medskip\noindent
                \textbf{Example \thethm .} }

\allowdisplaybreaks[4]
\numberwithin{equation}{thm}

\begin{document}
\setlength{\baselineskip}{4.9mm}
\setlength{\abovedisplayskip}{4.5mm}
\setlength{\belowdisplayskip}{4.5mm}
\renewcommand{\theenumi}{\roman{enumi}}
\renewcommand{\labelenumi}{(\theenumi)}
\renewcommand{\thefootnote}{\fnsymbol{footnote}}
\renewcommand{\thefootnote}{\fnsymbol{footnote}}
\parindent=20pt
\newcommand{\dis}{\displaystyle}

\medskip
\begin{center}
{\large \bf A cellular algebra with  certain idempotent decomposition } 
\\
\vspace{1cm}
Kentaro Wada\footnote{This research was  supported  in part by JSPS Research Fellowships for Young Scientists} 
\\ 
\vspace{0.5cm}
Graduate School of Mathematics \\
Nagoya University  \\
Chikusa-ku, Nagoya 464-8602,  Japan
\end{center}
\title{}
\maketitle
\markboth{Wada}{cellular algebra}


\begin{abstract}
For a cellular algebra $\A$ with a cellular basis $\ZC$, 
we consider  a decomposition of the unit element $1_\A$ into orthogonal idempotents (not necessary primitive) 
satisfying some conditions. 
By using this decomposition, 
the cellular basis $\ZC$ can be partitioned into some pieces with good properties. 
Then by using a certain map $\a$, we give a coarse partition of $\ZC$ whose refinement is the original partition. 
We construct a Levi type subalgebra $\aA$ of $\A$ and its quotient algebra $\oA$,  
and also construct a parabolic type subalgebra $\tA$ of $\A$, which contains $\aA$ 
with respect to the map $\a$. 
Then, we study the relation of standard modules, simple modules and  decomposition numbers among these algebras. 
Finally, we  study the relationship of blocks among these algebras.

\end{abstract}


\setcounter{section}{-1}
\section{Introduction}

A cellular algebra was introduced by Graham and Lehrer in \cite{GL96} 
with motivations from the representation theory of some algebras, 
such as Iwahori-Hecke algebras of type A, Ariki-Koike algebras, Brauer algebras and so on. 
Now, it is known that many classes of algebras are cellular, 
for example,  
cyclotomic $q$-Schur algebras associated to Ariki-Koike algebras \cite{DJM98}, 
Iwahori-Hecke algebras associated to the coxeter groups \cite{Geck08}, 
generalized $q$-Schur algebras \cite{Doty03}, 
Birman-Murakami-Wenzl algebras \cite{Eny04} and others. 

The general theory of cellular algebras in \cite{GL96} gives 
a systematic framework 
in (modular) representation theory of given algebras. 
Namely, for a cellular algebra $\A$ over a commutative ring $R$ which has a cellular basis $\ZC$ with respect to a poset $(\vL^+,\geq)$, 
one can define a standard module $W^\la$ ($\la \in \vL^+$), which is well understood  
by using the cellular basis $\ZC$, 
and one can define a bilinear form on $W^\la$ by using the multiplication of $\A$. 
Let $\rad W^\la$ be the radical of $W^\la$ with respect to this bilinear form, 
and put $L^\la=W^\la/\rad W^\la$. 
Then $\{L^\la \not=0\,|\,\la \in \vL^+\}$ gives a complete set of non-isomorphic simple $\A$-modules 
if $R$ ia a field. 
But, it is a difficult problem in general,  
not only to describe $\rad W^\la$ explicitly, 
but also 
to determine when $L^\la$ is non-zero. 
Hence, a fundamental problem is to compute the decomposition numbers of the given cellular algebra $\A$, 
which is the multiplicity of $L^\mu\not=0$ in the composition series of $W^\la$ ($\la,\mu \in \vL^+$).

In \cite{SW}, we gave a product formula for certain decomposition numbers of the cyclotomic $q$-Schur algebra $\Sc$. 
In [loc.\,cit.], we constructed a certain subalgebra $\Sp$ of $\Sc$ and its quotient algebra $\oSp$ by using the cellular basis of $\Sc$, 
and showed that $\oSp$ was isomorphic to a direct sum of tensor products of cyclotomic $q$-Schur algebras with smaller ranks than the original $\Sc$. 
The relation among $\Sc$, $\Sp$ and $\oSp$ combined with the structure of $\oSp$ 
implies the product formula for certain decomposition numbers of $\Sc$. 
(The original idea of them is due to Sawada \cite{Saw}, 
and our result in \cite{SW} is a generalization of his result.) 

In this paper, we generalize the above arguments for $\Sc$, 
except for the structure of $\oSp$, 
to a more general cellular algebra $\A$ with certain conditions. 
Instead of the combinatorics 
based on the explicit information on cellular basis 
in \cite{SW},  
we consider a decomposition of the unit element $1_\A$ of $\A$ into orthogonal idempotents (not necessary primitive) 
which  is 
compatible with a given cellular basis (Assumption \ref{assume}). 
By using this idempotent decomposition, 
the cellular basis $\ZC$ can be partitioned into some pieces with good properties. 
Then by using a certain map $\a$, 
we construct a coarse partition of $\ZC$ 
whose refinement is the original partition. 
By using this coarse partition of $\ZC$ obtained from the map $\a$, 
we construct a Levi type subalgebra $\aA$ of $\A$ and its quotient algebra $\oA$. 
We expect that $\oA$ has a simpler structure than $\A$ 
in analogy to $\oSp$ for the case of cyclotomic $q$-Schur algebra $\Sc$. 
Moreover, we construct a parabolic type subalgebra $\tA$ of $\A$, which contains $\aA$. 
$\oA$ is also obtained as a quotient algebra of $\tA$. 
In \cite{SW}, we only considered the parabolic type  subalgebra $\tSp$. 
However, 
considering the Levi type subalgebra $\aA$ gives 
a more useful information  than considering  only the parabolic type subalgebra $\tA$. 
Thus we have the following diagram. \\
\hspace{5em}
\begin{picture}(30,60)
\put(0,40){$\aA$ \scalebox{2}[1]{$\hookrightarrow $} $\tA$ \scalebox{2}[1]{$\hookrightarrow $} $\A$}
\put(50,0){$\oA$}
\put(5,10){\rotatebox{140}{\scalebox{4}[0.6]{$\twoheadleftarrow$}}}
\put(53,12){\rotatebox{90}{\scalebox{2.2}[1.2]{$\twoheadleftarrow$}}}
\end{picture}\vspace{2mm}\\
Each of $\aA$ and $\oA$ becomes a cellular algebra again.  
However, $\tA$ is not a cellular algebra but also is a standardly based algebra in the sence of \cite{DR98}. 
Moreover, we study the relation of standard modules, simple modules and decomposition numbers among $\A$, $\tA$, $\aA$ and $\oA$. 
As results, we see that a certain decomposition number of $\A$  coincides with a decomposition number of $\oA$. 

Finally, we study the relationship of blocks among $\A$, $\tA$, $\aA$ and $\oA$. 
We have not treated about blocks even in the case of the cyclotomic $q$-Schur algebra in \cite{SW} 
since we need the relation with the Levi type subalgebra $\aA$. 
By the general theory of cellular algebras, 
one can classify the 
blocks of a cellular algebra  by using the standard modules.  
Concerning the classification of blocks of $\tA$, the general theory can not be applied since $\tA$ is not a cellular algebra. 
However, we can classify the blocks of $\tA$ by using the right standard modules of $\tA$ thanks to the relation with $\aA$. 
By the relation of decomposition numbers among $\A$, $\tA$, $\aA$ and $\oA$, 
we see that the block decomposition of $\A$ is preserved in $\tA$. 
Namely, let $\A=\bigoplus_{\vG}\CB_\vG$ be the block decomposition of $\A$, 
then $\tA=\bigoplus_{\vG}\wt{\CB}_\vG^\a$  gives the block decomposition of $\tA$, 
where $\wt{\CB}_\vG^\a=\CB_\vG \cap \tA$. 
Moreover, let $\ol{\CB}_\vG^\a$ be the image of $\wt{\CB}_\vG^\a$ under the natural surjection $\tA \ra \oA$, 
then we have $\ol{\CB}_\vG^\a=\bigoplus_{\eta} \ol{\CB}_{\vG_\eta}^\a$ (this decomposition is determined by the map $\a$), 
and $\oA=\bigoplus_{\vG}\bigoplus_{\eta} \ol{\CB}_{\vG_\eta}^\a$ gives the block decomposition of $\oA$. \vspace{1em}\\
\textbf{Acknowledgment}\,\, I would like to thank Professor Toshiaki Shoji for many helpful advices and discussions.

\section{Cellular algebras}\label{cellular}
A cellular algebra was introduced by Graham and Lehrer in \cite{GL96}. 
In this section, 
we review its definition  and some fundamental properties. 
For further detailes on cellular algebras, 
one can refer to \cite{GL96} or a monograph \cite[\S2]{M-book}. 

\para \label{def-cellular}
Let $R$ be a commutative ring with $1$, and $\A$ be an associative algebra over $R$ with a unit element $1_\A$.  
Let $(\vL^+,\geq)$ be a finite poset, and $\CT(\la)$ be a finite index set for each $\la\in\vL^+$. 
For an ordered pair $(\Fs,\Ft) \in \CT(\la)\times \CT(\la)$ ($\la \in \vL^+$), 
we consider an element $c_{\Fs\Ft}^\la \in \A$, 
and a set  $\ZC=\{c_{\Fs\Ft}^\la \,|\,\Fs,\Ft\in \CT(\la) \text{ for some }\la \in \vL^+\}$. 
We say that $\A$ is a cellular algebra with a cellular basis $\ZC$ 
if $\A$ satisfies the  following properties. 
\begin{enumerate}
\item 
$\A$ has a free $R$-basis $\ZC=\{c_{\Fs\Ft}^\la \,|\,\Fs,\Ft\in \CT(\la) \text{ for some }\la \in \vL^+\}$. 
\item An $R$-linear map $* :\A\ra \A$, $a \mapsto a^\ast$, 
defined by $(c_{\Fs\Ft}^\la)^*=c_{\Ft\Fs}^\la$ for any $c_{\Fs\Ft}^\la \in \ZC$ 
gives an algebra anti-automorphism of $\A$.
\item For $c_{\Fs\Ft}^\la \in \ZC$ and $a\in \A$, 
\begin{align} \label{cellular-rel}
 c_{\Fs\Ft}^\la\cdot a \equiv \sum_{\Fv \in \CT(\la)} r_{\Fv}^{(\Ft,a)}c_{\Fs\Fv}^\la \mod \A^{\vee \la} \qquad (r_{\Fv}^{(\Ft,a)}\in R),
\end{align}
where $\A^{\vee \la}$ is an $R$-submodule of $\A$ spanned by 
$c_{\Fs'\Ft'}^{\la'}$ with $\la' \in \vL^+$ such that $\la'>\la$ and that $\Fs',\Ft'\in \CT(\la')$,  
and $r_{\Fv}^{(\Ft,a)}$ does not depend on the choice of $\Fs \in \CT(\la)$. 
\end{enumerate}

\para \label{def-rad}
By a general theory of cellular algebras, one can construct standard modules of $\A$, 
which have  properties  as follows. 

For $\la\in \vL^+$, let $W^\la$ be a free $R$-module with a free $R$-basis $\{c_{\Ft}^\la\,|\,\Ft\in \CT(\la)\}$.   
We define a right action of $\A$ on $W^\la$ by 
\[c_{\Ft}^\la\cdot a =\sum_{\Fv\in \CT(\la)}r_{\Fv}^{(\Ft,a)}\,c_{\Fv}^\la \qquad (\Ft\in \CT(\la),\,a \in \A), \] 
where $ r_{\Fv}^{(\Ft,a)}\in R$ is given in (\ref{cellular-rel}). 
We call a right $\A$-module $W^\la$ a standard module (or a cell module). 
We can define a bilinear form $\lan \,,\,\ran : W^\la \times W^\la \ra R $ by
\[\lan c_{\Fs}^\la, c_{\Ft}^\la \ran c_{\Fu\Fv}^\la \equiv c_{\Fu\Fs}^\la c_{\Ft\Fv}^\la \mod \A^{\vee \la} 
\quad \big(\Fs,\Ft,\Fu,\Fv \in \CT(\la)\big).\] 
Note that this definition does not depend on the choice of $\Fu,\Fv \in \CT(\la)$. 
Put 
\[\rad W^\la=\{x\in W^\la\,|\, \lan x,y\ran =0 \text{ for any }y \in W^\la \}.\]
Then $\rad W^\la$ turns out to be an $\A$-submodule of $W^\la$. Thus we can define a quotient module $L^\la=W^\la/\rad W^\la$. 
Set  
$\vL^+_0=\big\{\la\in \vL^+\bigm| L^\la \not= 0\big\}$.
The following result is fundamental in the theory of cellular algebras. 
\begin{thm}[{\cite[Theorem 3.4]{GL96}}]
Suppose that $R$ is a field. Then  
 $\{L^\la\,|\,\la \in \vL^+_0\}$ is a complete set of non-isomorphic right simple $\A$-modules. 
\end{thm}

\para
By applying the algebra anti-automorphism $\ast$ to both sides of (\ref{cellular-rel}), 
we have, for $c_{\Ft\Fs}^\la \in \ZC$ and $a^\ast \in \A$  
\begin{align} \label{cellular-rel-left}
 a^\ast \cdot c_{\Ft\Fs}^\la  \equiv \sum_{\Fv \in \CT(\la)} r_{\Fv}^{(\Ft,a)}c_{\Fv\Fs}^\la \mod \A^{\vee \la} \qquad (r_{\Fv}^{(\Ft,a)}\in R). 
\end{align}
Thus, for $\la \in \vL^+$, 
one can also define a left standard module $\shp W^\la$
with a free $R$-basis $\{\,\shp c^\la_{\Ft} \,|\,\Ft\in \CT(\la)\}$, 
where the left action of $\A$ on $\shp W^\la$ is given by 
\[a \cdot \,\sh c_{\Ft}^\la =\sum_{\Fv\in \CT(\la)}r_{\Fv}^{(\Ft,a^\ast)}\,\sh c_{\Fv}^\la \qquad (\Ft\in \CT(\la),\,a \in \A), \] 
where $ r_{\Fv}^{(\Ft,a^\ast)}\in R$ is given by replacing \lq\lq$a$" with \lq\lq$a^\ast$\," in (\ref{cellular-rel-left}). 
We can also define a bilinear form on $\shp W^\la$, 
a submodule $\rad\,\shp W^\la$ of $\shp W^\la$ 
and a quotient module $\shp L^\la=\,\shp W^\la/\rad \,\shp W^\la$ by a similar way as in the case of right standard modules $W^\la$. 
Moreover one  sees that $\shp L^\la \not=0$ for $\la \in \vL^+_0$. 
Then $\{\,\shp L^\la\,|\, \la \in \vL^+_0\}$ turns out to be a complete set of non-isomorphic left simple $\A$-modules 
when $R$ is a field. 

In the case where $R$ is a field, for $\la\in \vL^+$ and $\mu \in \vL^+_0$, 
let $[W^\la : L^\mu]_{\A}$ (resp. $[\,\shp W^\la :\,\shp L^\mu]_{\A}$) 
be a decomposition number of $\A$, 
namely a multiplicity of $L^\mu$ (resp. $\shp L^\mu$) in a composition series of $W^\la$ (resp. $\shp W^\la$). 
The following lemma is immediate from the definition. 

\begin{lem}\label{lr-aA}
Assume that $R$ is a field. For $\la \in \vL^+$ and $\mu \in \vL^+_0$, we have 
\[ [W^\la : L^\mu]_{\A} =[\,\shp W^\la:\,\shp L^\mu]_{\A}.\]
\end{lem}

\para \label{injective-standard-module}
We regard $\A$ itself as an $\A$-bimodule by the multiplication. 
Then $\A^{\vee \la}$ ($\la \in \vL^+$) is a two-sided ideal of $\A$ by (\ref{cellular-rel}) and (\ref{cellular-rel-left}). 
Thus $\A/\A^{\vee \la}$ is an $\A$-bimodule. 
By (\ref{cellular-rel}) and the definition of $W^\la$, 
we have an injective homomorphism of (right) $\A$-modules 
\begin{align}\label{injection-standard}
\phi_\la: W^\la \ra \A/\A^{\vee \la}
\end{align}
such that $\phi_\la(c_{\Ft}^\la)=c_{\Fs\Ft}^\la +\A^{\vee \la}$ ($\Ft \in \CT(\la)$), 
where $\Fs \in \CT(\la)$ is taken arbitrary, 
and we fix one of them. 
Thus we can regard $W^\la$ as an $\A$-submodule of $\A/\A^{\vee \la}$. 
For a left standard module $\shp W^\la$, 
we can regard $\shp W^\la$ as a (left) $\A$-submodule of $\A/\A^{\vee \la}$ in a similar way. 
This injection plays a fundamental role in later discussions. 

\para  \label{quotient of cellular}
Let $\wh{\vL}^+$ be a proper subset of $\vL^+$ satisfing the following condition. 
\begin{align} \label{saturated}
\text{If } \la >\mu \text{ for } \la \in \vL^+, \,\mu\in \wh{\vL}^+, 
\text{ then }\la\in \wh{\vL}^+. 
\end{align} 
We call $\wh{\vL}^+$ a \textbf{saturated subset} of $\vL^+$. 
Let $\A(\wh{\vL}^+)$ be an $R$-submodule of $\A$ spanned by 
$\{c_{\Fs\Ft}^\la\,|\,\Fs,\Ft \in \CT(\la) \text{ for some }\la \in \wh{\vL}^+ \}$. 
Then $\A(\wh{\vL}^+)$ is a two-sided ideal of $\A$ by the condition (\ref{saturated}) 
together with (\ref{cellular-rel}) and (\ref{cellular-rel-left}). 
Thus one can define a quotient algebra $\ol{\A}=\A/\A(\wh{\vL}^+)$. 
Set $\ol{\vL}^+=\vL^+ \setminus \wh{\vL}^+$. 
The following result is  clear from the definition. 

\begin{lem}\label{quotient-cellular}
$\ol{\A}$ is a cellular algebra with a cellular basis 
$\ol{\ZC}=\{\ol{c}_{\Fs\Ft}^\la\,|\, \Fs,\Ft\in \CT(\la) \text{ for}\\ \text{some }\la \in \ol{\vL}^+ \}$, 
where $\ol{c}_{\Fs\Ft}^\la$ is the image of $c_{\Fs\Ft}^\la$ under the natural surjection $\A \ra \A/\A(\wh{\vL}^+)$. 
\end{lem} 

Since $\ol{\A}$ is a cellular algebra, 
we can define a (right) standard module $\ol{W}^\la$ 
and its quotient module $\ol{L}^\la=\ol{W}^\la / \rad \ol{W}^\la $
for $\la \in \ol{\vL}^+$. 
Set $\ol{\vL}^+_0=\{\la \in \ol{\vL}^+\,|\,\ol{L}^\la \not=0\}$, 
then it is clear that $\ol{\vL}^+_0=\ol{\vL}^+\cap \vL^+_0$.  
Regarding an $\ol{\A}$-module $\ol{M}$ as an $\A$-module through the natural surjection $\A \ra \ol{\A}$,  
we have the following corollary. 
 
\begin{cor} \label{cor-quotient-cellular}
For $\la, \mu \in \ol{\vL}^+$, we have 
\begin{enumerate}
\item $\ol{W}^\la \cong W^\la$ as $\A$-modules. 
\item $\ol{L}^\mu \cong L^\mu$ as $\A$-modules. 
\item Suppose that $R$ is a field. 
Then we have
\[ [\ol{W}^\la : \ol{L}^\mu]_{\ol{\A}}=[W^\la : L^\mu]_{\A} \quad\text{ for }\la\in \ol{\vL}^+,\,\mu \in \ol{\vL}^+_0, \]
where $ [\ol{W}^\la : \ol{L}^\mu]_{\ol{\A}}$ is the decomposition number of $\ol{\A}$. 
\end{enumerate}
\end{cor}
\begin{proof}
Take $\la \in \ol{\vL}^+$. 
For $\nu \in \wh{\vL}^+$, we have $\la \not\geq \nu$ by the definition of $\wh{\vL}^+$ since $\la \not\in \wh{\vL}^+$. 
On the other hand, for $a \in \A$ and  $c_{\Fu\Fv}^\nu \in \ZC$, we have 
\begin{align*}\label{qqq}
a\cdot c_{\Fu\Fv}^\nu =\sum_{\nu' \geq \nu \atop \Fu', \Fv' \in \CT(\nu')} r_{\Fu'\Fv'}^{\nu'} c_{\Fu'\Fv'}^{\nu'} \qquad (r_{\Fu'\Fv'}^{\nu'} \in R)
\end{align*}
by (\ref{cellular-rel-left}). 
Thus, for $\nu \in \wh{\vL}^+$, 
we have  $c_{\Fs\Ft}^\la c_{\Fu\Fv}^\nu=0$ ($\Fs,\Ft \in \CT(\la), \,\Fu,\Fv \in \CT(\nu)$). 
This means that $W^\la \cdot \A(\wh{\vL}^+)=0$. 
It follows that the action of $\A$ on $W^\la$ induces the action of $\ol{\A}$ on $W^\la$, 
and this $\ol{\A}$-module $W^\la$ coincides with $\ol{W}^\la$. Thus we have (\roi). 
Now, (\roii) is clear since we have $\rad \ol{W}^\la=\rad W^\la$ from the definition.  
By (\roi),(\roii), a composition series of $\ol{W}^\la$ as $\ol{\A}$-modules turns out to be also a composition series as $\A$-modules. 
This proves (\roiii). 
\end{proof}

\para \label{cell-linked}
In the rest of this section, we assume that $R$ is a field. 
By \cite{GL96}, a classification of blocks of cellular algebra $\A$ has been obtained by using standard modules. 
Here, we recall their results. 

We say that an $\A$-module $M$ belongs to a block $B$ of $\A$ if all of its composition factors lie in $B$. 
It is known by \cite{GL96} that, for each $\la \in \vL^+$, all of composition factors of $W^\la$ lie in the same block. 
Thus $W^\la$ belongs to some block of $\A$. 

For $\la,\mu \in \vL^+$, we say that $\la$ and $\mu$ are (cell) linked, and 
denote by $\la \eqA \mu$  
if there exists a sequence $\la=\la_1,\la_2,\cdots ,\la_k=\mu$ $(\la_i \in \vL^+)$ 
such that $W^{\la_i}$ and $W^{\la_{i+1}}$ have a common composition factor for $i=1,\cdots,k-1$. 
Then the relation $\eqA$ gives an equivalence relation on $\vL^+$. 
Thus, if $\la \eqA \mu$ then $W^\la$ and $W^\mu$ belong to the same block. 
Moreover, it is known that the converse also holds,  
namely we have the following proposition. 

\begin{prop}[{\cite{GL96}}]\ 
For $\la,\mu \in \vL^+$, 
$\la \eqA \mu$ if and only if $W^\la$ and $W^\mu$ belong to the same block of $\A$. 
\end{prop}

Let $\vL^+/_\sim$ be the set of equivalence classes on $\vL^+$ 
with respect to the relation $\eqA$. 
As a corollary of the above proposition, we have the following classification of blocks of $\A$. 

\begin{cor}\ \label{block-class}
There is a one-to-one correspondence between $\vL^+/_\sim $ 
and the set of blocks of $\A$.
\end{cor}

\remark 
We can also consider  an equivalence relation $\eqAl$  on $\vL^+$   
by using the left standard modules of $\A$  in a similar way. 
For $\la ,\mu \in \vL^+$, we have 
\begin{align}
\la \eqA \mu \quad \text{ if and only if } \quad \la \eqAl \mu, 
\end{align}
by Lemma \ref{lr-aA}. 
Thus $\vL^+/_\sim$ coincides with the set of equivalent classes on $\vL^+$ 
with respect to the  relation $\eqAl$. 


\para
For $\vG \in \vL^+/_\sim$, 
we denote by $\CB_\vG$ the block of $\A$ corresponding to $\vG$ 
under the bijection in Corollary \ref{block-class}. 
Then we have a block decomposition of $\A$; 
\begin{align}
\A=\bigoplus_{\vG\in \vL^+/_\sim}\CB_\vG.
\label{block-decom}
\end{align} 

Let $e_\vG$ be the block idempotent such that $e_\vG\A=\CB_\vG$ for $\vG\in \vL^+/_\sim$, 
and $e_\vG=e_1+e_2+\cdots+e_k$ be the decomposition of $e_\vG$ 
into the primitive idempotents. 
It is known that there exists $j$ such that $e=e_j$ 
if a primitive idempotent $e$ of $\A$ is equivalent to $e_i$ for some $i=1,\cdots,k$, 
namely if $e\A \cong e_i \A$ as right $\A$-modules. 
Moreover, we know that, 
for any primitive idempotent $e$ of $\A$, 
$e^\ast$ is equivalent to $e$ by \cite{GL96}. 
This implies that $(e_\vG)^\ast=e_\vG$. 
Thus, $\CB_\vG=e_\vG \A e_\vG$ turns out to be a cellular algebra with respsect to the involution $\ast$ by \cite[Proposition 4.3]{KX98}. 
 
In general, the cellular basis $\ZC$ of $\A$ is not compatible with the block decomposition. 
Namely, $\ZC \cap \CB_\vG$ does not give a basis of $\CB_\vG$. 
However, one can take another cellular basis in the following way,
 which is compatible with the block decomposition.

\begin{prop} \label{basis-block}
Let the block decomposition of $\A$ be as in $(\ref{block-decom})$.  
Then $\CB_{\vG}$ has a cellular basis 
$\ZC_\vG=\{b_{\Fs\Ft}^\la\,|\,\Fs,\Ft\in \CT(\la) \text{ for some }\la \in \vG\}$ 
such that $b_{\Fs\Ft}^\la \equiv c_{\Fs\Ft}^\la \mod \A^{\vee \la}$ in $\A$. 
Moreover, $\bigcup\limits_{\vG \in \vL^+/_\sim} \ZC_\vG$ gives a cellular basis of $\A$. 
\end{prop}

\begin{proof}
Suppose that $|\vL^+|=k$, and express the set $\vL^+$ as   
$\vL^+=\{\la_1,\la_2,\cdots,\la_k\}$ 
so that 
if $\la_i  > \la_j$ then $i<j$. 
Let $\A_{\leq i}$ be an $R$-submodule of $\A$ spanned by 
$\{c_{\Fs\Ft}^{\la_j}\,|\,\Fs,\Ft\in \CT(\la_j),\,1\leq j \leq i\}$. 
Then we have the following sequence of two-sided ideals of $\A$. 
\begin{align}
\A=\A_{\leq k} \supset \A_{\leq k-1} \supset \cdots \supset \A_{\leq 1} \supset \A_{\leq_0}=0,
\label{filtration of A}
\end{align}
where $ \A_{\leq i}\big/ \A_{\leq i-1} \cong \,\sh W^{\la_i}\otimes W^{\la_i}$ as  $(\A,\A)$-bimodules. 

For a right $\A$-module $M$, the $R$-submodule $M\cdot e_{\vG}$ of $M$ is invariant under the action of $\A$ 
since $e_\vG$ is contained in the center of $\A$. 
Moreover, for an exact sequence of right $\A$-modules 
\[0 \ra N \ra M \ra M/N \ra 0,\]
we have an exact sequence of right $\A$-modules 
\[0 \ra N\cdot e_\vG \ra M\cdot e_\vG \ra (M/N)\cdot e_\vG \ra 0.\]
Hence, the filtration (\ref{filtration of A}) induces a filtration of $(\A,\A)$-bimodules,  
\begin{align}
\A e_\vG =\A_{\leq k} e_\vG \supset \A_{\leq k-1} e_\vG \supset \cdots \supset \A_{\leq 1} e_\vG \supset \A_{\leq_0} e_\vG=0,
\label{filtration of Ae}
\end{align} 
where $(\A_{\leq i}\cdot e_{\vG}) \big/ (\A_{\leq i-1}\cdot e_\vG) \cong \,\sh W^{\la_i}\otimes W^{\la_i} \cdot e_\vG$ as $(\A,\A)$-bimodules, 
and similarly for $e_\vG \A$ and $e_\vG \A e_\vG$. 
In particular, we have a filtration  
\begin{align}
e_\vG \A e_\vG= e_\vG \A_{\leq k} e_\vG  \supset e_\vG \A_{\leq k-1} e_\vG 
	\supset \cdots \supset e_\vG \A_{\leq 1} e_\vG \supset e_\vG \A_{\leq_0} e_\vG=0,
\label{filtration of eAe}
\end{align}
where $(e_\vG \cdot \A_{\leq i} \cdot e_\vG )\big/ (e_\vG \cdot \A_{\leq i-1}\cdot e_\vG) 
	\cong e_\vG \cdot\,\hspace{-1mm}\sh W^{\la_i}\otimes W^{\la_i}\cdot e_\vG$ 
as $(\A,\A)$-bimodules. 
Since $e_\vG \A_{\leq i} e_\vG= \A_{\leq i} e_\vG = e_\vG \A_{\leq i}$ for any $i=1,\cdots,k$, 
we have 
\begin{align}
e_\vG \cdot \,\hspace{-1mm} \sh W^{\la_i} \otimes W^{\la_i} \cdot e_\vG 
=e_\vG \cdot \,\hspace{-1mm} \sh W^{\la_i} \otimes W^{\la_i}  
=\sh W^{\la_i} \otimes W^{\la_i} \cdot e_\vG.
\label{eWe=eW=We}
\end{align}
Note that $W^{\la_i}$ (resp. $\sh W^{\la_i}$) belongs to the block $\CB_\vG$ if and only if $\la_i\in \vG$, 
then (\ref{eWe=eW=We}) implies that 
\[W^{\la_i}\cdot e_\vG = \begin{cases} W^{\la_i} &\text{ if }\la_i \in \vG, \\ 0 &\text{ otherwise}, \end{cases}
\hspace{3em}
e_\vG \cdot \,\hspace{-1mm} \sh W^{\la_i} = \begin{cases} \sh W^{\la_i} &\text{ if }\la_i \in \vG, \\ 0 &\text{ otherwise}. \end{cases}\] 
Moreover, (\ref{filtration of eAe}) gives a cell chain of $\CB_\vG=e_\vG \A e_\vG$ 
in the sense of \cite[\S3]{KX98}. 
Now, the cellular basis $\ZC_\vG$ of $\CB_\vG$ satisfing the proposition  
is obtained from this cell chain (\ref{filtration of eAe}) by using the argument in \cite[\S3]{KX98}. 
\end{proof}

\remark 
\begin{enumerate}
\item 
The proof of Proposition \ref{basis-block} gives an alternative proof of the fact that 
all the composition factors of $W^\la$ (resp. $\sh W^\la$) lie in the same block, since 
(\ref{eWe=eW=We}) implies  that 
$W^{\la_i}\cdot e_\vG$ (resp. $ e_\vG \cdot \,\hspace{-1mm} \sh W^{\la_i} $) 
is isomorphic to $W^{\la_i}$ (resp. $\sh W^{\la_i}$) or zero.

\item
For $\la,\mu \in \vL^+$, $\la \eqA \mu$ 
if and only if 
$W^\la \cdot e =W^\la$ and $W^{\mu} \cdot e =W^\mu$ 
for the same block idempotent $e$ of $\A$. 
\end{enumerate}


\section{A Levi type subalgebra  and  its quotient algebra} \label{aA-and-oA}
In this section, we construct a Levi type subalgebra $\aA$ of $\A$ and its  quotient algebra $\oA$. 
We expect that $\oA$ will have simpler structures than $\A$. (See the example for cyclotomic $q$-Schur algebras.) 

In order to construct a subalgebra $\aA$ of $\A$ and its quotient algebra $\oA$, 
we pose the  following assumption for a cellular algebra $\A$ and its cellular basis $\ZC$.  

\begin{ass}\label{assume}
Throughout the remainder of this paper, we assume the following statements $(A1)-(A4)$. 
\begin{description}

\item[(A1)] 
There exists a poset $(\wt{\vL}, \geq)$ 
such that $\vL^+ \subset \wt{\vL}$  
and that the partial order on $\wt{\vL}$ is an extention of  the partial order on $\vL^+$.  

\item[(A2)] 
There exists a subset $\vL \subset \wt{\vL}$, 
and the unit element $1_\A$ of $\A$ is decomposed as 
\[1_\A=\sum_{\mu \in \vL}e_\mu \quad \text{with } e_\mu\not=0 \text{ and }e_\mu e_\nu =e_\nu e_\mu=\d_{\mu \nu}e_\mu \quad (\mu,\nu \in \vL),\] 
where $\d_{\mu\nu}=1$ if $\mu=\nu$, and $\d_{\mu\nu}=0$ if $\mu \not= \nu$. 
Thus $\{e_\mu\,|\,\mu\in \vL\}$ is a set of pairwise orthogonal idempotents.
\item[(A3)] 
For $\mu \in \vL$, we have 
\begin{align}
e_{\mu}= \hspace{-1em}\sum_{\la \in \vL^+ \text{ s.t. }\la \geq \mu \atop \Fs, \Ft \in \CT(\la)} 
\hspace{-1em}r_{\Fs \Ft}^\la c_{\Fs \Ft}^\la \qquad(r_{\Fs\Ft}^\la \in R). \label{linear-com-emu}
\end{align}
\item[(A4)] 
For each $\la \in \vL^+$ and each $\Ft \in \CT(\la)$, there exists an element $\mu \in \vL$ such that 
\begin{align}
&c_{\Fs\Ft}^\la e_\mu=c_{\Fs\Ft}^\la &\hspace{-5em}\text{for any }\Fs\in \CT(\la), \label{assume-emu1}\\
&e_\mu c_{\Ft\Fu}^\la =c_{\Ft\Fu}^\la &\hspace{-5em}\text{for any }\Fu\in \CT(\la). \label{assume-emu2}   
\end{align}
\end{description}
\end{ass}

The following lemma  is obtained from the above assumption immediately. 
\begin{lem} \label{unique-lem}
Let $\Ft \in \CT(\la)$ $\big(\la \in \vL^+\big)$ and $\mu \in \vL$ be 
such that $c_{\Fs\Ft}^\la e_\mu=c_{\Fs\Ft}^\la$ $\big(\Fs \in \CT(\la)\big)$. 
Then for any $\nu \in \vL$ such that $\nu \not=\mu$, 
we have $c_{\Fs\Ft}^\la e_\nu=0$. 
In particular, for each $\Ft\in \CT(\la)$, there exists a unique $\mu\in \vL$ such that 
$c_{\Fs\Ft}^\la e_\mu=c_{\Fs\Ft}^\la$ and $e_\mu c_{\Ft\Fs}^\la =c_{\Ft\Fs}^\la$ for any $\Fs \in \CT(\la)$. 
\end{lem}
\begin{proof}
For $\nu \in \vL$ such that $\nu\not=\mu$, $e_\nu$ and $e_\mu$ are orthogonal. Thus we have 
\[c_{\Fs\Ft}^\la e_\nu=(c_{\Fs\Ft}^\la e_\mu)e_\nu=0.\]
Now, the latter statement is clear by (A4). 
\end{proof}

\exam\label{ex}
Throughout the paper, we consider the following three algebras as examples. 

(\roi) (Matrix algebra)
Let $\A=M_{n\times n}(R)$ be an $n\times n$ matrix algebra over $R$. We take $\vL^+=\{n\}$ and $\CT(n)=\{1,2,\cdots ,n\}$. 
For $i,j\in \CT(n)$, let $E_{ij}$ be the elementary matrix having $1$ at the $(i,j)$-entry and $0$ elsewhere. 
Then $\ZC=\{E_{ij}\,|\,i,j\in \CT(n) \}$ becomes a cellular basis of $\A$. 
The unit element of $\A$ is $1_\A=\sum_{i\in \CT(n)}E_{ii}$. 
Put $\wt{\vL}=\vL=\CT(n)$. 
We consider the order on $\wt{\vL}$ by the natural order on the set of integers. 
Clearly, this setting satisfies Assumption \ref{assume}. 
\vspace{2mm}

(\roii) (Path algebra) 
Let $Q$ be the quiver \\
\hspace{5em}
\begin{picture}(100,30)
\put(10,8){$1 \qquad \,\,2 \qquad \,\,3 \qquad \,\,4 \qquad\,\,5,$}
\put(16,20){$ \underrightarrow{\,\,\a_{12}\,\,}$}
\put(16,0){$ \overleftarrow{\,\,\a_{21}\,\,}$}
\put(50,20){$ \underrightarrow{\,\,\a_{23}\,\,}$}
\put(50,0){$ \overleftarrow{\,\,\a_{32}\,\,}$}
\put(84,20){$ \underrightarrow{\,\,\a_{34}\,\,}$}
\put(84,0){$ \overleftarrow{\,\,\a_{43}\,\,}$}
\put(118,20){$ \underrightarrow{\,\,\a_{45}\,\,}$}
\put(118,0){$ \overleftarrow{\,\,\a_{54}\,\,}$}
\end{picture}\\
and $\CI$ be the two-sided ideal of the path algebra $RQ$ generated by 
\begin{align*}
&\a_{12}\a_{23},\,\, \a_{23}\a_{34},\,\, \a_{34}\a_{45},\,\, \a_{54}\a_{43},\,\, \a_{43}\a_{32},\,\, \a_{32}\a_{21}, \\
&\a_{21}\a_{12}-\a_{23}\a_{32},\,\,\a_{32}\a_{23}-\a_{34}\a_{43},\,\,\a_{43}\a_{34}-\a_{45}\a_{54}, 
\end{align*}
where we denote by $\a_{12}\a_{23}$ 
the path $1 \stackrel{\a_{12}\,}{\longrightarrow } 2 \stackrel{\a_{23}\,}{\longrightarrow } 3$, etc.. 

Let $\A=RQ/\CI$. 
We define the algebra anti-automorphism $\ast: a \mapsto a^\ast$ of $RQ$ by 
$\a_{ij}^\ast=\a_{ji}$, 
then this induces the algebra anti-automorphism of $\A$, which is denoted by $\ast$ again. 
It is easy to check, by using the defining relations, that  
$\A$ has a free $R$-basis 
\begin{align}
\a_{12}\a_{21}, \quad\,\,
\begin{matrix} e_1, & \a_{12},\\ \a_{21}, & \a_{21}\a_{12}, \end{matrix} \quad\,\,
\begin{matrix} e_2, & \a_{23},\\ \a_{32}, & \a_{32}\a_{23}, \end{matrix} \quad\,\,
\begin{matrix} e_3, & \a_{34},\\ \a_{43}, & \a_{43}\a_{34}, \end{matrix} \quad\,\,
\begin{matrix} e_4, & \a_{45},\\ \a_{54}, & \a_{54}\a_{45}, \end{matrix} \quad\,\,
e_5,
\label{quiver-basis}
\end{align}
where $e_i$ is the path of length $0$ on the vertex $i$. 
(We use the relations,  $\a_{23}\a_{32}=\a_{21}\a_{12}$, $\a_{12}\a_{21}\a_{12}=\a_{12}\a_{23}\a_{32}=0$, etc..)

Let $\vL^+=\{\la_0,\la_1,\cdots,\la_5\}$ with the order $\la_0>\la_1>\cdots>\la_5$, 
and put 
$\CT(\la_0)=\{1\}$, $\CT(\la_i)=\{i,\,i+1\}$ $(1\leq i \leq 4)$, $\CT(\la_5)=\{5\}$. 
We define 
\[c_{11}^{\la_0}=\a_{12}\a_{21},\quad  
\left(\begin{matrix} c_{ii}^{\la_i} & c_{i\,i+1}^{\la_i}\\ c_{i+1\,i}^{\la_i} & c_{i+1\,i+1}^{\la_i} \end{matrix}\right)=
\left(\begin{matrix} e_i & \a_{i\,i+1}\\ \a_{i+1\,i} & \a_{i+1\,i}\,\a_{i\,i+1} \end{matrix} \right)\,\,(1\leq i \leq 4), \quad
c_{55}^{\la_5}=e_5. 
\]
Then $\ZC=\{c_{ij}^{\la}\,|\,i,j \in \CT(\la) \text{ for some }\la \in \vL^+\}$ 
is a cellular basis of $\A$. 

Let $\wt{\vL}=\{\la_0,\la_1,\cdots,\la_5,1,2,3,4,5\}$ and $\vL=\{1,2,3,4,5\}$, 
then we have $1_\A=\sum_{i\in \vL}e_i$. 
We define the order on $\wt{\vL}$ by 
$\la_0>\la_1>1>\la_2>2> \cdots >\la_5>5$. 
Clearly, this setting satisfies Assumption \ref{assume}.
(In the later discussions, 
we will denote the basis of $\A$ by the form (\ref{quiver-basis}) rather than by $\ZC$.)
\vspace{2mm}

(\roiii) (Cyclotomic $q$-Schur algebra) 
Let $\He=\He_{n,r}$ be an Ariki-Koike algebra over $R$ associated to the complex reflection group $\FS_n \ltimes (\ZZ/r\ZZ)^n$ 
with parameters $q,Q_1,\cdots,Q_r \in R$. 
Let $\vL$ be the set of $r$-compositions $\Bmu=(\mu^{(1)},\cdots,\mu^{(r)})$ of size $n$ 
such that each composition $\mu^{(i)}$ has the length at most $n$, 
and $\vL^+$ be the set of $r$-partitions $\Bla=(\la^{(1)},\cdots,\la^{(r)})$ of size $n$. 
We define the partial order on $\vL$, 
the so-called \lq\lq dominance order", by 
\[\Bla \geq \Bmu \,\,\text{ if and only if } \,\,
\sum_{i=1}^{l-1}|\la^{(i)}|+\sum_{j=1}^k\la^{(l)}_j \geq \sum_{i=1}^{l-1}|\la^{(i)}|+\sum_{j=1}^k\la^{(l)}_j \]
for any $1\leq l \leq r-1$, $1\leq k \leq n$. 
We define the cyclotomic $q$-Schur algebra associated to $\He$ by 
\[\Sc=\Sc(\vL)=\End_{\He} \Big( \bigoplus_{\Bmu \in \vL}M^{\Bmu} \Big),\]
where $M^{\Bmu}$ is a certain right $\He$-module attached to $\Bmu \in \vL$ 
(see \cite{DJM98} or \cite{SW} for further detailes). 
By \cite{DJM98}, $\Sc$ has a cellular basis 
\[\ZC=\big\{\vf_{ST}\bigm| S,T\in \CT_0(\Bla) \text{ for some }\Bla \in \vL^+ \big\},\]
where $\CT_0(\Bla)=\bigcup_{\Bmu \in \vL} \CT_0(\Bla,\Bmu)$, and 
$\CT_0(\Bla,\Bmu)$ is the set of semi-standard tableaux with  shape $\Bla$ and  type $\Bmu$. 
If $S \in \CT_0(\Bla,\Bmu),\,T\in \CT_0(\Bla,\Bnu)$ for $\Bla \in \vL^+, \Bmu,\Bnu \in \vL$, 
then $\vf_{ST} \in \Hom_\He(M^{\Bnu},M^{\Bmu})$ and $\vf_{ST}(M^{\Bt})=0$ for  $\Bt \in \vL$ such that $\Bt \not=\Bnu$. 
For $\Bmu \in \vL$, let $\vf_{\Bmu}$ be the identity map on $M^{\Bmu}$ and zero on $M^{\Bt}$ for $\Bt \in \vL$ such that $\Bt\not=\Bmu$. 
Then we have $\vf_{\Bmu}\in \Sc$ and  $1_{\Sc}=\sum_{\Bmu \in \vL}\vf_{\Bmu}$. 
In this case, we have $\wt{\vL}=\vL \supset \vL^+$. 
One see that this setting satisfies Assumption \ref{assume}. 

\remarks 
\begin{enumerate}
\item
It may occur that $\wt{\vL}=\vL$. (e.g. the case of cyclotomic $q$-Schur algebra.)

\item
It is not necessary to assume that  $e_\mu$ is  primitive for $\mu \in \vL$. 

\item
If $e_\mu^*=e_\mu$ for any $\mu \in \vL$ 
then (\ref{assume-emu2}) is deduced from (\ref{assume-emu1}) 
since $e_\mu c_{\Ft\Fs}^\la =e_\mu^* c_{\Ft\Fs}^\la=(c_{\Fs\Ft}^\la e_\mu)^*=(c_{\Fs\Ft}^\la)^*=c_{\Ft\Fs}^\la$. 
\end{enumerate}

\para
In view of (A4), we define, 
for $\la\in \vL^+$ and $\mu\in \vL$ 
\begin{align*}
\CT(\la,\mu)&=\big\{\Ft \in \CT(\la)\,|\, c_{\Fs\Ft}^\la e_{\mu}=c_{\Fs\Ft}^\la \,\,\text{ for any }\Fs \in \CT(\la) \big \}\\
&=\{\Fs \in \CT(\la)\,|\, e_\mu c_{\Fs\Ft}^\la =c_{\Fs\Ft}^\la \,\,\text{ for any }\Ft \in \CT(\la) \}. 
\end{align*}
By Lemma \ref{unique-lem}, we have a disjoint union 
\[\CT(\la)=\coprod_{\mu\in \vL}\CT(\la,\mu).\]

\begin{lem}\label{e-lem}
For $\mu \in \vL$, we have 
\begin{align}
e_{\mu}= \sum_{\la \in \vL^+ \text{ s.t. }\la\geq \mu \atop \Fs,\Ft\in \CT(\la,\mu)}r_{\Fs\Ft}^\la c_{\Fs\Ft}^\la . 
\label{re-linear-com-emu}
\end{align}
\end{lem}
\begin{proof}
By multiplying $e_\mu$ on both sides of (\ref{linear-com-emu}) from right and left, we obtain the lemma. 
\end{proof}
Next lemma plays an important role in later discussions. 

\begin{lem}\label{key-lem}\
\begin{enumerate}
\item
If $\CT(\la,\mu) \not= \emptyset$ 
for $\la\in \vL^+$ and $\mu \in \vL$, 
then $\la \geq \mu$.  
\item
For $\Fs_i \in \CT(\la_i,\mu_i)$ and $\Ft_i\in \CT(\la_i,\nu_i)$ $(i=1,2)$, we have 
\[c_{\Fs_1\Ft_1}^{\la_1} c_{\Fs_2 \Ft_2}^{\la_2}=\d_{\nu_1 \mu_2}
	\hspace{-1.5em}\sum_{\la \geq \la_1 \text{ and }\la_2 \atop \Fs\in \CT(\la,\mu_1),\,\Ft\in \CT(\la,\nu_2)} 
	\hspace{-1.5em}r_{\Fs\Ft}^\la\, c_{\Fs\Ft}^\la \qquad (r_{\Fs\Ft}^\la \in R). \]
\end{enumerate}
\end{lem}

\begin{proof}
(\roi) For $\Ft\in \CT(\la,\mu)$, we have $c_{\Fs\Ft}^\la e_\mu=c_{\Fs\Ft}^\la$. 
On the other hand, 
since we can express 
$e_{\mu}=\sum r_{\Fs'\Ft'}^{\la'}c_{\Fs'\Ft'}^{\la'}$ 
by (A3),  
where $\la'\geq \mu$ and $\Fs',\Ft' \in \CT(\la')$, 
we have $c_{\Fs\Ft}^\la e_\mu =\sum f_{\Fs'\Ft'}^{\la'}c_{\Fs'\Ft'}^{\la'}$ 
by (\ref{cellular-rel-left}) 
with $\la'\geq \mu$ and $\Fs',\Ft' \in \CT(\la')$. 
Thus we have $\la \geq \mu$ if $\CT(\la,\mu)\not=\emptyset$. 

Next we prove (\roii). If $\nu_1 \not=\mu_2$ then we have 
\begin{align}
c_{\Fs_1\Ft_1}^{\la_1}c_{\Fs_2\Ft_2}^{\la_2} &=\big(c_{\Fs_1\Ft_1}^{\la_1} e_{\nu_1}\big) \big(e_{\mu_2} c_{\Fs_2\Ft_2}^{\la_2}\big)
\label{prod-basis-zero}\\
	&=0. \notag
\end{align}
We now suppose that $\nu_1 =\mu_2$. By (\ref{cellular-rel}) and (\ref{cellular-rel-left}), 
one can write as 
\[ c_{\Fs_1\Ft_1}^{\la_1}c_{\Fs_2\Ft_2}^{\la_2}= \sum_{\la \geq \la_1 \text{ and }\la_2 \atop \Fs,\Ft\in \CT(\la)}r_{\Fs\Ft}^\la c_{\Fs\Ft}^\la
	\qquad (r_{\Fs\Ft}^\la \in R). \]
Thus we have 
\begin{align*}
c_{\Fs_1\Ft_1}^{\la_1}c_{\Fs_2\Ft_2}^{\la_2}&=\big(e_{\mu_1} c_{\Fs_1\Ft_1}^{\la_1}\big)\big(c_{\Fs_2\Ft_2}^{\la_2} e_{\nu_2}\big) \\
&=\sum_{\la \geq \la_1 \text{ and }\la_2 \atop \Fs,\Ft\in \CT(\la)}r_{\Fs\Ft}^\la e_{\mu_1}c_{\Fs\Ft}^\la e_{\nu_2}\\ 
&=\hspace{-1.5em}\sum_{\la \geq \la_1 \text{ and }\la_2 \atop \Fs\in \CT(\la,\mu_1),\,\Ft\in \CT(\la,\nu_2)} 
	\hspace{-1.5em}r_{\Fs\Ft}^\la\, c_{\Fs\Ft}^\la . 
\end{align*}
\end{proof}

\para \label{take-alpha}
In order to divide $\CT(\la)$ ($\la \in \vL^+)$ into a disjoint union of two subsets of $\CT(\la)$, 
we consider a poset $(X,\succeq )$ and a map $\a:\wt{\vL} \ra X$ such that   
\begin{align} \label{condition-alpha}
\a(\la)\succeq \a(\mu) \,\,\text{ if }\,\, \la \geq \mu  \,\,\text{ for }\,\, \la,\mu \in \wt{\vL}. 
\end{align}
Then we have a poset $\big(\{\a(\la)\,|\,\la \in \wt{\vL}\}, \succeq \big)$. 
Note that it may happen $\a(\la)=\a(\mu)$ even if $\la \not= \mu$.

\exam\label{ex2} 
We consider three algebras in Example \ref{ex}, 
and define a map $\a: \wt{\vL}\ra X$  satisfing the condition (\ref{condition-alpha}) 
for each case.  

(\roi) (Matrix algebra) 
Fix an integer $b$ such that $1 < b \leq n$. 
Put $X=\{l,s\}$, and define the order on $X$ by $l \succ s$. 
We define a map $\a: \wt{\vL}\ra X$ by 
\[\a(k)=\begin{cases} l &\text{ if }k\geq b\\ s &\text{ if }k<b \end{cases} \qquad \big(k\in \wt{\vL}=\{1,2,\cdots,n\}\big).\]
Then this map $\a$ satisfies the condition (\ref{condition-alpha}). 

(\roii) (Path algebra) 
Put $X=\{t_0, t_{123}, t_{45}\}$, and define the order on $X$ by $t_0 \succ t_{123} \succ t_{45}$. 
We define the map $\a: \wt{\vL}\ra X$ by 
\[ \a(\la_k)=\begin{cases} t_0 &\text{ if }k =0 \\ t_{123} &\text{ if }k=1,2,3 \\ t_{45} &\text{ if }k=4,5 \qquad,\end{cases} \quad 
\a(k)=\begin{cases} t_{123} &\text{ if }k=1,2,3\\ t_{45} &\text{ if }k=4,5 \qquad. \end{cases}\]
Then this map $\a$ satisfies the condition (\ref{condition-alpha}). 

(\roiii) (Cyclotomic $q$-Schur algebra) 
Fix an integer $g$ ($1< g <r$) and 
choose an $r$-turple $\Bp=(r_1,\cdots,r_g)\in \ZZ_{>0}^g$ such that $r_1+\cdots+r_g=r$. 
Put $p_k=r_1+\cdots+r_{k-1}$ for $k=1,\cdots,g$ with $p_1=0$. 
For $\Bmu=(\mu^{(1)},\cdots,\mu^{(r)})\in \vL$, 
put  $\Bmu^{[k]}=(\mu^{(p_k+1)},\cdots,\mu^{(p_k+r_k)})$ for $k=1,\cdots,g$. 
Let $X=\ZZ_{\geq 0}^g$ be the poset endowed with the lexicographic order  \lq\lq $\succeq$".   
We define a map $\a_{\Bp}: \vL \ra X$ by 
\[ \a_{\Bp}(\Bmu)=(|\Bmu^{[1]}|,|\Bmu^{[2]}|,\cdots,|\Bmu^{[g]}|) \quad (\Bmu \in \vL).\]
Then this map $\a_{\Bp}$ satisfies the condition (\ref{condition-alpha}). 


\para 
From now on, we fix a poset $\big(\{\a(\la)\,|\,\la \in \wt{\vL}\} ,\succeq \big)$ 
as in \ref{take-alpha},  
and we define an $R$-submodule $\aA$ of $\A$ with respect to the map $\a$. 
We shall prove that $\aA$ is a subalgebra of $\A$, and that $\aA$ inherits a cellular structure from $\A$. 

For $\la \in \vL^+$, set 
\[\CT^+_\a(\la)=\coprod_{\mu \in \vL\atop \a(\la)=\a(\mu)}\CT(\la,\mu), \qquad \CT^-_\a(\la)=\coprod_{\mu\in \vL \atop \a(\la)\succ \a(\mu)}\CT(\la,\mu).\] 
(We denote $\CT^{\pm}_\a(\la)$ by $\CT^{\pm}(\la)$ simply, if there is no fear of confusion.)   
By Lemma \ref{key-lem} and (\ref{condition-alpha}), we have $\CT(\la,\mu)\not= \emptyset \Ra \la\geq \mu \Ra \a(\la)\succeq \a(\mu)$. Thus we have 
\[\CT(\la)=\coprod_{\mu\in \vL}\CT(\la,\mu)=\CT^+_\a(\la) \cup \CT^-_\a(\la) \quad(\text{disjoint union}).\] 

For $(\la,\ve)\in \vL^+ \times \{0,1\}$, set 
\[I(\la,\ve)=\begin{cases} \CT^+_\a(\la) &\text{if }\ve=0\\ \CT^-_\a(\la) &\text{if }\ve=1. \end{cases} \]
Note that it may happen that $\CT^+_\a(\la)$ or $\CT^-_\a(\la)$  is an empty set. 
Then we set   
\[\Om =\big(\vL^+ \times \{0,1\}\big) \setminus \Big( \{(\la,0)\,|\,\CT^+_\a(\la)=\emptyset \}\cup \{(\la,1)\,|\,\CT^-_\a(\la)=\emptyset \}\Big)\]
so that $I(\la,\ve)\not=\emptyset$ for any $(\la,\ve) \in \Om$. 
We define a partial order on $\Om$ by 
\begin{align}
(\la,\ve)>(\la',\ve') \,\,\text{ if }\,\,\ve>\ve' \text{ or if } \ve=\ve' \text { and }\la>\la'. 
\label{def-order-Om}
\end{align}
For $(\la,\ve)\in \Om$, set 
\[\ZC^\a(\la,\ve)=\big\{c_{\Fs\Ft}^\la \bigm| \Fs,\Ft \in I(\la,\ve) \text{ for some }(\la,\ve)\in \Om\big\},\]
and put $\ZC^\a(\la)=\ZC^\a(\la,0)\cup \ZC^\a(\la,1)$ for $\la \in \vL^+$, 
where  we define $\ZC^\a(\la,\ve)=\emptyset$ 
if $(\la,\ve)\not\in \Om$. 
We define a subset $\ZC^\a$ of $\ZC$ by 
\begin{align*}
\ZC^\a&=\big\{c_{\Fs\Ft}^\la \bigm| \Fs,\Ft\in I(\la,\ve) \text{ for some }(\la,\ve)\in \Om \big\}\\
&=\coprod_{(\la,\ve)\in \Om}\ZC^\a(\la,\ve).
\end{align*}

We define an $R$-submodule $\A^\a$ of $\A$ as the $R$-span of $\ZC^\a$. 
In order to see that $\A^\a$ becomes a subalgebra of $\A$, we prepare the following lemma. 

\begin{lem}\label{pro-lem}
For $ c_{\Fs_1\Ft_1}^{\la_1}, c_{\Fs_2\Ft_2}^{\la_2} \in \ZC^\a$ 
such that $\Fs_i \in \CT(\la_i,\mu_i)$ and $\Ft_i\in \CT(\la_i,\nu_i)$ $(i=1,2)$, 
we have 
\[ c_{\Fs_1\Ft_1}^{\la_1}c_{\Fs_2\Ft_2}^{\la_2} =\begin{cases}
	\dis \sum_{c_{\Fs,\Ft}^{\la_1}\in \ZC^\a(\la_1,0)}r_{\Fs\Ft}^{\la_1}c_{\Fs\Ft}^{\la_1} 
		+ \sum_{\la>\la_1}\sum_{c_{\Fs,\Ft}^{\la}\in \ZC^\a(\la)}r_{\Fs\Ft}^{\la}c_{\Fs\Ft}^{\la}, 
		&\text{if } c_{\Fs_1 \Ft_1}^{\la_1}\in \ZC^\a(\la_1,0), \\\,\\
	\dis \sum_{\la\geqq\la_1}\sum_{c_{\Fs,\Ft}^{\la}\in \ZC^\a(\la,1)}r_{\Fs\Ft}^{\la}c_{\Fs\Ft}^{\la}, 
		&\text{if } c_{\Fs_1 \Ft_1}^{\la_1}\in \ZC^\a(\la_1,1),\\\,\\
	\dis \sum_{c_{\Fs,\Ft}^{\la_2}\in \ZC^\a(\la_2,0)}r_{\Fs\Ft}^{\la_2}c_{\Fs\Ft}^{\la_2} 
		+ \sum_{\la>\la_2}\sum_{c_{\Fs,\Ft}^{\la}\in \ZC^\a(\la)}r_{\Fs\Ft}^{\la}c_{\Fs\Ft}^{\la}, 
		&\text{if } c_{\Fs_2 \Ft_2}^{\la_2}\in \ZC^\a(\la_2,0), \\\,\\
	\dis \sum_{\la\geqq\la_2}\sum_{c_{\Fs,\Ft}^{\la}\in \ZC^\a(\la,1)}r_{\Fs\Ft}^{\la}c_{\Fs\Ft}^{\la},
		&\text{if } c_{\Fs_2 \Ft_2}^{\la_2}\in \ZC^\a(\la_2,1).
\end{cases}\]
\end{lem}

\begin{proof}
By Lemma \ref{key-lem}, we may suppose $\nu_1=\mu_2$. 
Then we have following four cases. 

\begin{enumerate}

\item $c_{\Fs_1 \Ft_1}^{\la_1}\in \ZC^\a(\la_1,0)$ and $c_{\Fs_2 \Ft_2}^{\la_2}\in \ZC^\a(\la_2,0)$; 
In this case, we have $\a(\la_1)=\a(\mu_1)=\a(\nu_1)=\a(\mu_2)=\a(\la_2)=\a(\nu_2)$. 

\item $c_{\Fs_1 \Ft_1}^{\la_1}\in \ZC^\a(\la_1,0)$ and $c_{\Fs_2 \Ft_2}^{\la_2}\in \ZC^\a(\la_2,1)$; 
In this case, we have  $\a(\la_1)=\a(\mu_1)=\a(\nu_1)=\a(\mu_2)$, $\a(\la_2)\succ \a(\mu_2)$ and $\a(\la_2)\succ \a(\nu_2)$. 
Hence $\a(\la_2) \succ \a(\la_1)$. 

\item $c_{\Fs_1 \Ft_1}^{\la_1}\in \ZC^\a(\la_1,1)$ and $c_{\Fs_2 \Ft_2}^{\la_2}\in \ZC^\a(\la_2,0)$; 
In this case, we have $\a(\la_1)\succ\a(\mu_1)$, $\a(\la_1)\succ\a(\nu_1)=\a(\mu_2)$ and $\a(\la_2)= \a(\mu_2)=\a(\nu_2)$. 
Hence $\a(\la_1) \succ \a(\la_2)$. 

\item $c_{\Fs_1 \Ft_1}^{\la_1}\in \ZC^\a(\la_1,1)$ and $c_{\Fs_2 \Ft_2}^{\la_2}\in \ZC^\a(\la_2,1)$; 
In this case, we have $\a(\la_1)\succ\a(\mu_1)$, $\a(\la_1)\succ\a(\nu_1)$, $\a(\la_2)\succ \a(\mu_2)$ and $\a(\la_2)\succ \a(\nu_2)$. 
\end{enumerate}

Note that $\a(\la) \succeq \a(\la_1)$ if $\la\geq \la_1$ \big(resp. $\a(\la) \succeq \a(\la_2)$ if  $\la\geq \la_2$ \big), then 
by applying Lemma \ref{key-lem} (\roii) to each of the above four cases, 
we obtain the lemma. 
\end{proof}

By this lemma, we  see that $\A^\a$ is a subalgebra of $\A$. 
Moreover  $\A^\a$ inherits a cellular structure from $\A$, namely 
we have the following theorem. 

\begin{thm} \label{aA-thm}
$\A^\a$ is a subalgebra of $\A$ containing the unit element $1_\A$ of $\A$. Moreover, 
$\A^\a$ turns out to be a cellular algebra with a cellular basis $\ZC^\a$ 
with respect to the poset $\Om$ 
and the index set $I(\la,\ve)$ for $(\la,\ve)\in \Om$,  
i.e.; the following property holds; 
\begin{enumerate}
\item 
An $R$-linear map $\ast : \aA \ra \aA$ defined by 
$c_{\Fs\Ft}^\la \mapsto c_{\Ft\Fs}^\la $ for $c_{\Fs\Ft}^\la \in \ZC^\a$ 
gives an algebra anti-automorphism of $\A^\a$. 
\item For any $a \in \A^\a,\,c_{\Fs\Ft}^\la \in \ZC^\a(\la,\ve)$, 
\begin{align}
c_{\Fs\Ft}^\la \cdot a \equiv \sum_{\Fv\in I(\la,\ve) }r_{\Fv}^{(\Ft,a)} c_{\Fs\Fv}^{\la} 
\mod (\A^\a)^{\vee (\la,\ve)}\qquad (r_{\Fv}^{(\Ft,a)}\in R),  
\label{cellular-rel-aA}
\end{align}
where $(\A^\a)^{\vee(\la,\ve)}$ is an $R$-submodule of $\A^\a$ spanned by 
$\big\{c_{\Fs'\Ft'}^{\la'}\in \ZC^\a(\la',\ve')\,|\\ (\la',\ve')\in \Om \text{ such that }(\la',\ve')>(\la,\ve)\big\}$, 
and $r_{\Fv}^{(\Ft,a)}$ does not depend on  the choice of $\Fs \in I(\la,\ve)$. 
\end{enumerate}
\end{thm}

\begin{proof}
We have already seen that $\A^\a$ is a subalgebra of $\A$. 
(\roi) follows from the fact the map $\ast$ on $\A$ leaves $\A^\a$ invariant. 
(\roii) follows from Lemma \ref{pro-lem} combined with cellular structure of $\A$. 
Thus, we have only to show  that $\A^\a$ contains $1_\A$. 
Recall that $1_\A=\sum_{\mu\in \vL}e_\mu$ by (A-2). 
By Lemma \ref{e-lem}, 
$e_\mu$ is written as a linear combination of $c_{\Fs\Ft}^\la$ 
with $\a(\la)\succeq \a(\mu)$, where $\Fs,\Ft\in \CT(\la,\mu)$. 
Hence $c_{\Fs\Ft}^\la \in \ZC^\a(\la)$ and 
we have $1_\A\in \A^\a$. 
\end{proof}

\para
We apply the general theory of cellular algebras to $\A^\a$. 
For $(\la,\ve)\in \Om$, let $Z^{(\la,\ve)}$ be a standard (right) $\A^\a$-module. 
Thus $Z^{(\la,\ve)}$ has an $R$-free basis 
$\{c_{\Ft}^{(\la,\ve)}\,|\, \Ft\in I(\la,\ve)\}$. 
We can consider the symmetric bilinear form  $\lan\,,\,\ran_\a: Z^{(\la,\ve)}\times Z^{(\la,\ve)}\ra R$ defined by the equation 
\[\lan c_{\Fs}^{(\la,\ve)} , c_{\Ft}^{(\la,\ve)}\ran_\a \,c_{\Fu\Fv}^{\la} \equiv c_{\Fu\Fs}^{\la}c_{\Ft\Fv}^{\la} \mod (\A^\a)^{\vee (\la,\ve)} 
\quad (\Fs,\Ft,\Fu,\Fv \in I(\la,\ve)).\] 
The radical $\rad Z^{(\la,\ve)}$ of $Z^{(\la,\ve)}$ with respect to $\lan\,,\,\ran_\a$ 
is defined as in \ref{def-rad}, and  
we define the quotient module 
$L^{(\la,\ve)}=Z^{(\la,\ve)}/ \rad Z^{(\la,\ve)}$. 
Set 
\[\Om_0=\big\{(\la,\ve)\in \Om\bigm| L^{(\la,\ve)}\not=0\big\}.\]
Then we have the following corollary. 

\begin{cor}
Assume that $R$ is a field. Then  $\{L^{(\la,\ve)}\,|\, (\la,\ve) \in \Om_0\}$ 
 is a complete set of non-isomorphic right simple $\A^\a$-modules. 
\end{cor}

\remark \label{remark-relation-A-aA}
In view of Lemma \ref{pro-lem}, 
one can replace $(\aA)^{\vee (\la,\ve)}$ in (\ref{cellular-rel-aA}) by $(\aA\cap \A^{\vee \la})$. 
By using this fact, 
we have the following property. 

\begin{enumerate} 
\item  
For $\Ft,\Fv \in I(\la,\ve)$ and $a \in \aA$, 
the coefficient $r_{\Fv}^{(\Ft,a)}$ in (\ref{cellular-rel-aA}) 
is equal to  $r_{\Fv}^{(\Ft,a)}$ in (\ref{cellular-rel}). 

\item 
For any $\Fs,\Ft \in I(\la, \ve)$ $\big((\la,\ve)\in \Om \big)$, 
we have 
\[ \lan c_{\Fs}^{(\la,\ve)}, c_{\Ft}^{(\la,\ve)} \ran_\a= \lan c_{\Fs}^\la , c_{\Ft}^\la \ran.\]   
\end{enumerate}

\para \label{quotient-alg}
Next, we construct a quotient algebra of $\aA$. 
We take a subset $\wh{\Om}=\{(\la,\ve)\in \Om\,|\,\ve=1\}$ of $\Om$. 
Then one  sees that $\wh{\Om}$ turns out to be a saturated subset of $\Om$ (cf. (\ref{saturated}))
by  definition of the partial order (\ref{def-order-Om}) on $\Om$.   

Let $\aA(\wh{\Om})$ be an $R$-submodule of $\aA$ spanned by 
$c_{\Fs \Ft}^\la$, where $\Fs,\Ft \in I(\la,1)$ for some $(\la,1)\in \wh{\Om}$. 
Then by  \ref{quotient of cellular}, 
$\aA(\wh{\Om})$ is a two-sided ideal, and 
one can define a quotient algebra $\oA=\aA/ \aA(\wh{\Om})$. 

Set $\ol{\vL}^+=\vL^+ \setminus \big\{\la \in \vL^+ \bigm| \CT^+(\la)=\emptyset \big\}$. 
Then there exists a bijection 
between $\ol{\Om}=\Om \setminus \wh{\Om}$ and $\ol{\vL}^+$ 
by the assignment $\ol{\Om}\ni(\la,0) \leftrightarrow  \la \in \ol{\vL}^+$. 
(Note that $\ol{\Om}=\{(\la,\ve)\in \Om\,|\,\ve=0\}$.)
Set 
\[\ol{\ZC}^\a=\big\{\ol{c}_{\Fs\Ft}^\la \bigm| \Fs,\Ft \in \CT^+(\la) \text{ for some }\la \in \ol{\vL}^+\big\},\]
where $\ol{c}_{\Fs\Ft}^\la$ is the image of $c_{\Fs\Ft}^\la$ under the natural surjection $\pi:\aA \ra \oA$. 
(Note that $I(\la,0)=\CT^+(\la)$ for $\la \in \vL^+$.)
We have the following result by Lemma \ref{quotient-cellular}. 

\begin{thm}
$\oA$ is a cellular algebra with a cellular basis $\ol{\ZC}^\a$. 
\end{thm}

\para
We apply the general theory of cellular algebras to $\oA$. 
For $\la\in \ol{\vL}^+$, let $\ol{Z}^{\la}$ be a standard (right) $\oA$-module. 
Thus $\ol{Z}^{\la}$ has an $R$-free basis 
$\{\ol{c}_{\Ft}^{\la}\,|\, \Ft\in \CT^+(\la)\}$. 
We can consider the symmetric bilinear form  
$\lan \,,\,\ran_{\ol\a}:\ol{Z}^{\la}\times \ol{Z}^{\la}\ra R$ defined by 
\[\lan \ol{c}_{\Fs}^\la,\ol{c}_{\Ft}^\la \ran_{\ol{\a}} \,\ol{c}_{\Fu\Fv}^\la \equiv \ol{c}_{\Fu\Fs}^\la \ol{c}_{\Ft\Fv}^\la \mod (\oA)^{\vee \la} 
\quad (\Fs,\Ft,\Fu,\Fv \in \CT^+(\la)).\]   
The radical $\rad \ol{Z}^\la$ of $\ol{Z}^\la$ with respect to $\lan \,,\,\ran_{\ol{\a}}$ 
is defined as in \ref{def-rad}, and 
we define the quotient module 
$\ol{L}^{\la}=\ol{Z}^{\la}/ \rad \ol{Z}^{\la}$. 
Set 
\[\ol{\vL}^+_0=\big\{\la \in \ol{\vL}^+ \bigm| \ol{L}^\la \not=0\big\}.\]
Then we have the following corollary. 

\begin{cor}
Assume that $R$ is a field. Then  $\{\ol{L}^{\la}\,|\, \la \in \ol{\vL}^+_0\}$ 
 is a complete set of non-isomorphic right simple $\oA$-modules. 
\end{cor}

\remarks \label{remark-relation-A-aA-oA}
\begin{enumerate}
\item
For $\Fs, \Ft \in \CT^+(\la)$ $(\la \in \ol{\vL}^+)$, we have 
\[\lan \ol{c}_{\Fs}^\la , \ol{c}_{\Ft}^\la \ran_{\ol{\a}}=  \lan c_{\Fs}^{(\la,0)}, c_{\Ft}^{(\la,0)} \ran_\a= \lan c_{\Fs}^\la , c_{\Ft}^\la \ran.\]
\item
By (\roi), we  see that $\ol{\vL}^+_0=\Om_0 \cap \ol{\Om}$ under the bijection 
$\ol{\Om} \stackrel{\sim}{\ra} \ol{\vL}^+$ in \ref{quotient-alg}. 
\item
If $\CT^+(\la) \not= \emptyset$ for any $\la \in \vL^+$, then we have 
$\ol{\vL}^+=\vL^+$. 
\end{enumerate}


\exam 
\label{ex3}
We give the Levi type subalgebra and its quotient algebra for each case in Example \ref{ex} 
with respect to the map $\a$ ($\a_{\Bp}$ in (\roiii)) in Example \ref{ex2}. 

(\roi) (Matrix algebra) 
By definition, we have $\CT_\a^+(n)=\{b,b+1,\cdots,n\}$, $\CT_\a^-(n)=\{1,2,\cdots,b-1\}$. 
Thus we have 
$\ZC^\a=\{E_{ij}\,|\, i,j\in \CT_\a^{\pm}(n)\}$ and $\ol{\ZC}^\a=\{ \ol{E}_{ij}\,|\,i,j \in \CT^+(n)\}$. 
Hence we have, 
\[ \A^\a={\footnotesize \left( \begin{array}{ccc|ccc}
	* & \cdots &* &  & & \\
	\vdots& \ddots& \vdots&&\scalebox{2}{0}&\\
	* & \cdots &* &  & & \\\hline
	 &  & & * & \cdots &* \\
		&\scalebox{2}{0}& &\vdots&\ddots&\vdots\\
	 && & * & \cdots &* \\
	\end{array} \right),} \qquad 
\oA={\footnotesize \left( \begin{array}{ccc|ccc}
	&  & &  & & \\
	& \scalebox{2}{0}&&&\scalebox{2}{0}&\\
	 & & &  & & \\\hline
	 &  & & * & \cdots &* \\
		&\scalebox{2}{0}& &\vdots&\ddots&\vdots\\
	 && & * & \cdots &* \\
	\end{array} \right).} 
\vspace{-1em}
\]
\hspace{7.5em}$\underbrace{\hspace{4.0em}}_{b-1}$
\hspace{0.5em}$\underbrace{\hspace{4.0em}}_{n-b+1}$
\hspace{7.0em}$\underbrace{\hspace{3.0em}}_{b-1}$
\hspace{0.5em}$\underbrace{\hspace{4.0em}}_{n-b+1}$
\vspace{5mm}

(\roii) (Path algebra) 
By definition, 
we have $\CT^-(\la_0)=\CT(\la_0)$, $\CT^+(\la_3)=\{3\}$, $\CT^-(\la_3)=\{4\}$ and $\CT^+(\la_i)=\CT(\la_i)$ for $i=1,2,4,5$. 
Thus, 
\begin{align*}
\ZC^\a= \left\{\a_{12}\a_{21}, \quad
\begin{matrix} e_1, & \hspace{-2mm} \a_{12},\\ \a_{21}, & \hspace{-2mm} \a_{21}\a_{12}, \end{matrix} \quad
\begin{matrix} e_2, & \hspace{-2mm} \a_{23},\\ \a_{32}, & \hspace{-2mm} \a_{32}\a_{23}, \end{matrix} \quad
\begin{matrix} e_3, & \hspace{-2mm} \\  & \hspace{-2mm} \a_{45}\a_{54}, \end{matrix} \quad
\begin{matrix} e_4, & \hspace{-2mm} \a_{45},\\ \a_{54}, & \hspace{-2mm} \a_{54}\a_{45}, \end{matrix} \quad
e_5\right\},\\ \\
\ol{\ZC}^\a= \left\{\qquad\,\,  \qquad
\begin{matrix} \ol{e_1}, & \hspace{-2mm} \ol{\a_{12}},\\ \ol{\a_{21}}, & \hspace{-2mm} \ol{\a_{21}\a_{12}}, \end{matrix} \quad
\begin{matrix} \ol{e_2}, & \hspace{-2mm} \ol{\a_{23}},\\ \ol{\a_{32}}, & \hspace{-2mm} \ol{\a_{32}\a_{23}}, \end{matrix} \quad
\begin{matrix} \ol{e_3}, & \hspace{-2mm} \\  & \qquad\quad \hspace{-2mm}  \end{matrix} \quad
\begin{matrix} \ol{e_4}, & \hspace{-2mm} \ol{\a_{45}},\\ \ol{\a_{54}}, & \hspace{-2mm} \ol{\a_{54}\a_{45}}, \end{matrix} \quad
\ol{e_5}\right\},  
\end{align*}
where note that $\a_{43}\a_{34}=\a_{45}\a_{54}$ in $\A$. 
Moreover, we have 
\[\Om=\{(\la_0,1),(\la_3,1),(\la_i,0)\,|\,1\leq i \leq 5\}\,  
\text{ and } \,
\ol{\vL}^+=\{\la_i\,|\, 1\leq i \leq5\}.
\] 

We see that  
\[\aA\cong (RQ_1^\a / \CI_1^\a) \oplus (RQ_2^\a / \CI_2^\a), \]
where \\
\begin{picture}(100,30)
\put(30,8){$Q_1^\a=$}
\put(60,8){$\Big(1 \qquad \,\,2 \qquad \,\,3\Big),$}
\put(73,20){$ \underrightarrow{\,\,\a_{12}\,\,}$}
\put(73,0){$ \overleftarrow{\,\,\a_{21}\,\,}$}
\put(107,20){$ \underrightarrow{\,\,\a_{23}\,\,}$}
\put(107,0){$ \overleftarrow{\,\,\a_{32}\,\,}$}
\put(180,8){$Q_2^\a=$}
\put(210,8){$\Big(4 \qquad \,\,5\Big),$}
\put(223,20){$ \underrightarrow{\,\,\a_{45}\,\,}$}
\put(223,0){$ \overleftarrow{\,\,\a_{54}\,\,}$}
\end{picture}\vspace{-1em}\\
\[\CI_1^\a= \lan \a_{12}\a_{23},\,\a_{32}\a_{21},\, \a_{21}\a_{12}-\a_{23}\a_{32}\ran_{\text{ideal}}, \quad 
\CI_2^\a=\lan \a_{45}\a_{54}\a_{45},\, \a_{54}\a_{45}\a_{54}\ran_{\text{ideal}},\]
and we have 
\[\oA\cong (R\ol{Q}_1^\a / \ol{\CI}_1^\a) \oplus (R\ol{Q}_2^\a / \ol{\CI}_2^\a), \]
where 
$\ol{Q}_1^{\a}=Q_1^{\a}$, $\ol{Q}_2^{\a}=Q_2^\a$, 
\[\ol{\CI}_1^\a= \lan \a_{12}\a_{21}, \a_{12}\a_{23},\,\a_{32}\a_{21},\, \a_{21}\a_{12}-\a_{23}\a_{32}\ran_{\text{ideal}}, \quad   
\ol{\CI}_2^\a=\lan \a_{45}\a_{54}\ran_{\text{ideal}}.\] 
Note that all of $(RQ_i^\a / \CI_i^\a)$ and  $(R\ol{Q}_i^\a / \ol{\CI}_i^\a)$ ($i=1,2$) 
are cellular algebras. 
\vspace{3mm}

(\roiii) (Cyclotomic $q$-Schur algebra) 
Fix $\Bp=(r_1,\cdots,r_g)\in \ZZ_{>0}^g$ such that $r_1+\cdots+r_g=r$. 
Put $X^+=\{\a_{\Bp}(\mu)\,|\, \mu \in \vL\}$. 
We denote by $\Sp$ the Levi type subalgebra, 
and denote by $\oSp$ the quotient algebra of $\Sp$ 
with respect to the map $\a_{\Bp}$. 
By \cite[Theorem 4.15]{SW}, we have 
\begin{align}
\oSp \cong \bigoplus_{(n_1,\cdots,n_g)\in X^+} \Sc(\vL_{n_1,r_1})\otimes \cdots \otimes \Sc(\vL_{n_g,r_g}) 
\quad \text{ as algebras},
\label{decom-oSp}
\end{align}
where $\vL_{n_k,r_k}=\{\Bmu^{[k]}\,|\,\Bmu\in \vL \text{ such that }|\Bmu^{[k]}|=n_k\}$, 
and $\Sc(\vL_{n_k,r_k})$ is the cyclotomic $q$-Schur algebra associated to $\He_{n_k,r_k}$ with parameters 
$q, Q_{p_k+1},\cdots,Q_{p_k+r_k}$. 

However we note that there is a discrepancy in the notation between \cite{SW} and here. 
Although the quotient algebra $\oSp$ in our notation is the same as in \cite{SW}, 
the  $\Sp$ in \cite{SW} does not coincide with our $\Sp$. 
Rather it coincides with the parabolic type subalgebra $\tSp$ which will be discussed in Section \ref{tA}.

\section{Relations among $\A,\aA$ and $\oA$}

Summing up the arguments in the previous section, we have the following diagram.
{\center
$\aA \scalebox{1.5}[1]{$\hookrightarrow$} \A$\vspace{-4mm}\\
\hspace{-15mm}\rotatebox{270}{\scalebox{1.8}[1]{$\twoheadrightarrow $}}\vspace{0mm}\\
\hspace{-12mm}$\oA$\\}
In this section, we study some relations among $\A$,$\aA$ and $\oA$ for standard modules, simple modules, and decomposition numbers. 
For convenience sake, 
we define $Z^{(\la,\ve)}=0$ for $(\la,\ve)\in \vL^+\times\{0,1\}\setminus \Om$, 
and we define $\ol{Z}^\la=0$ for $\la \in \vL^+ \setminus \ol{\vL}^+$. 
\para 
First, we study the relation for standard modules and simple modules between $\A$ and $\aA$. 
Recall that, for $\la \in \vL^+$, $I(\la,0)=\CT^+(\la)$, $I(\la,1)=\CT^-(\la)$ and $\CT(\la)=\CT^+(\la)\cup \CT^-(\la)$. 
We also recall that  $Z^{(\la,\ve)}$ has an $R$-free basis $\{c_{\Ft}^{(\la,\ve)}\,|\,\Ft\in I(\la,\ve)\}$ ($\ve=0,1$), 
and $W^\la$ has an $R$-free basis $\{c_{\Ft}^\la\,|\,\Ft\in \CT(\la)\}$. 

By the argument in \ref{injective-standard-module}, 
for $\la \in \vL^+$, 
there exists an injective homomorphism of $\A$-modules 
\[\phi_\la : W^\la \ra \A/\A^{\vee \la}\] 
such that  $\phi_\la(c_{\Ft}^\la)=c_{\Fs\Ft}^\la +\A^{\vee \la}$ $\big(\Ft \in \CT(\la)\big)$, 
where $\Fs \in \CT(\la)$ is taken arbitrarily and is fixed. 
Similarly, 
for $(\la,\ve)\in \Om$, 
there exists an injevtive homomorphism of $\aA$-modules 
\[\phi_{(\la,\ve)}: Z^{(\la,\ve)} \ra \aA\big/\big(\aA\cap \A^{\vee \la}\big)\]
such that $ \phi_{(\la,\ve)}(c_{\Ft}^{(\la,\ve)})=c_{\Fs\Ft}^{\la} +(\aA\cap \A^{\vee \la})$ $\big(\Ft \in I(\la,\ve)\big)$, 
where $\Fs \in I(\la,\ve)$ is taken arbitrarily and is fixed. 
(Note Remark \ref{remark-relation-A-aA}.)

By regarding $\A$-modules as $\aA$-modules by restriction, we have the following proposition.  

\begin{prop}\label{aA-A}
For $\la \in \vL^+$, we have the following properties. 
\begin{enumerate}

\item 
There exists an isomorphism 
$f_\la:Z^{(\la,0)}\oplus Z^{(\la,1)} \stackrel{\sim}{\ra}W^\la $ of $\aA$-modules 
such that $(c_{\Ft}^{(\la,0)},0)\mapsto c_{\Ft}^\la$ $\big( \Ft\in \CT^+(\la)\big)$ 
and $(0,c_{\Ft}^{(\la,1)})\mapsto c_{\Ft}^\la$ $\big( \Ft \in \CT^-(\la)\big)$. 

\item 
$L^\la \cong L^{(\la,0)} \oplus L^{(\la,1)}$ as $\aA$-modules. 
\end{enumerate}
\end{prop}
\begin{proof}
(\roi) 
Since $\A^\a /(\aA \cap \A^{\vee \la})$ is an $\A^\a$-submodule of $\A/\A^{\vee\la}$,  
we have a  homomorphism of $\A^\a$-modules 
\[ (\phi_{(\la,0)}\oplus\phi_{(\la,1)}) : Z^{(\la,0)} \oplus Z^{(\la,1)} \ra \A/\A^{\vee\la} \]
by $(\phi_{(\la,0)}\oplus \phi_{(\la,1)})(x\oplus y)=\phi_{(\la,0)}(x)+\phi_{(\la,1)}(y)$. 
By chasing the image of the bases of $Z^{(\la,0)}$ aand $Z^{(\la,1)}$, 
we see that $(\phi_{(\la,0)}\oplus \phi_{(\la,1)})$ is injective,  
and that the image of $(\phi_{(\la,0)}\oplus\phi_{(\la,1)})$ coincides with the image of $\phi_\la$. 
This implies (\roi). 

Next we prove (\roii). 
By Remark \ref{remark-relation-A-aA}, 
we have $\lan c_{\Fs}^{(\la,\ve)},c_{\Ft}^{(\la,\ve)}\ran_\a=\lan c_{\Fs}^\la, c_{\Ft}^\la \ran$ for any $\Fs,\Ft\in I(\la,\ve)$ $(\ve=0,1)$.  
Moreover, for any $\Fs \in I(\la,0),\,\Ft \in I(\la,1)$, 
we have $c_{\Fs' \Fs}^\la c_{\Ft \Ft'}^\la=0$ $\big(\Fs',\Ft'\in \CT(\la)\big)$ 
by a similar reason as in (\ref{prod-basis-zero}). 
It follows that $\lan c_{\Fs}^\la, c_{\Ft}^\la \ran =0$ 
for any $\Fs \in I(\la,0),\,\Ft \in I(\la,1)$. 
Similarly, we have  $\lan c_{\Fu}^\la, c_{\Fv}^\la \ran =0$ 
for any $\Fu \in I(\la,1),\,\Fv \in I(\la,0)$. 
These imply that $\rad W^\la \cong \rad Z^{(\la,0)}\oplus \rad Z^{(\la,1)}$ under the isomorphism $W^\la \cong Z^{(\la,0)}\oplus Z^{(\la,1)}$
 as $\aA$-modules. Thus we have $W^\la/\rad W^\la\cong Z^{(\la,0)}/\rad Z^{(\la,0)} \oplus Z^{(\la,1)}/\rad Z^{(\la,1)}$ 
as $\aA$-modules. This proves (\roii).  
\end{proof}

Next, we study the relation for standard modules and simple modules between $\aA$ and $\oA$.  
When we regard $\oA$-modules as $\aA$-modules through the natural surjection $\pi: \aA \ra \oA$, 
we have the following proposition by Corollay \ref{cor-quotient-cellular}.

\begin{prop}\label{aA-oA}
For $\la\in \ol{\vL}^+$, we have the following properties. 
\begin{enumerate}
\item 
The map defined by 
$c_{\Ft}^{(\la,0)} \mapsto \ol{c}_{\Ft}^\la$ $\big(\Ft\in \CT^+(\la)\big)$ 
gives an isomorphism 
$g_\la: Z^{(\la,0)}\stackrel{\sim}{\ra} \ol{Z}^\la$ of $\aA$-modules. 
\item
$L^{(\la,0)}\cong L^\la$ as $\aA$-modules. 
\end{enumerate}
\end{prop}

\para
In the rest of this section, we assume that $R$ is a field. 
We consider the decomposition numbers
\begin{align*}
&[W^\la:L^\mu]_\A\quad \hspace{10mm}\big(\la\in \vL^+,\,\,\mu \in \vL^+_0\big),\\
&[Z^{(\la,\ve)}:L^{(\mu,\ve')}]_{\aA}\quad \big((\la,\ve)\in \Om,\, \,(\mu,\ve')\in \Om_0\big),\\
&[\ol{Z}^{\la}:\ol{L}^{\mu}]_{\oA}\quad \hspace{10mm}\big(\la \in \ol{\vL}^+,\,\, \mu\in \ol{\vL}^+_0\big),
\end{align*}
for $\A,\aA$ and $\oA$ respectively. 
Note that, for $\mu \in \vL^+$, if $L^\mu=0$ then we have $L^{(\mu,0)}=0$ and $L^{(\mu,1)}=0$, 
and if $L^\mu\not=0$ then at least one of $L^{(\mu,0)}$ or $L^{(\mu,1)}$ is non-zero 
by Proposition \ref{aA-A} (\roii). 
We have the following proposition on the relation for decomposition numbers between $\aA$ and $\oA$.  

\begin{prop} \label{decom-prop}
Assume that $R$ is a field. 
Then the following properties hold. 
\begin{enumerate}
\item
For $(\la,0)\in \Om,\,(\mu,0)\in \Om_0$, we have $[Z^{(\la,0)}:L^{(\mu,0)}]_{\aA}=[\ol{Z}^\la:\ol{L}^\mu]_{\oA}$.
\item 
For $(\la,0)\in \Om,\,(\mu,1)\in \Om_0$, we have $[Z^{(\la,0)}:L^{(\mu,1)}]_{\aA}=0$. 
\item
For $(\la,1)\in \Om,\,(\mu,0)\in \Om_0$ such that $\a(\la)=\a(\mu)$, we have $[Z^{(\la,1)}:L^{(\mu,0)}]_{\aA}=\nobreak 0$.
\item
For $(\la,0)\in \Om,\,(\mu,0)\in \Om_0$ such that $\a(\la)\not=\a(\mu)$, we have $[Z^{(\la,0)}:L^{(\mu,0)}]_{\aA}=\nobreak 0$. 
\item 
For $\la \in \ol{\vL}^+,\,\mu \in \ol{\vL}_0^+$ such that $\a(\la)\not=\a(\mu)$, we have $[\ol{Z}^\la :\ol{L}^\mu]_{\oA}=0$. 
\end{enumerate}
In particular, a simple $\A^\a$-module occuring in a composition series of $Z^{(\la,0)}$ 
is always of the form $L^{(\mu,0)}$ for some $(\mu,0)\in \Om_0$. 
 
\end{prop}

\begin{proof}
(\roi),(\roii) 
First we note that $\mu \in \ol{\vL}^+_0$ if $(\mu,0)\in \Om_0$ by Remark \ref{remark-relation-A-aA-oA} (\roii). 
Take a composition series of $\ol{Z}^\la$ as  $\oA$-modules. 
If we regard this composition series as a filtration of  $\aA$-modules through the natural surjection $\pi: \aA\ra \oA$, 
this filtration turns out to be a composition series of $Z^{(\la,0)}$ as  $\aA$-modules by Proposition \ref{aA-oA}. 
This implies (\roi) and (\roii). 

(\roiii)
Let $Z^{(\la,1)}=M_k\supsetneqq M_{k-1}\supsetneqq \cdots \supsetneqq M_1\supsetneqq M_0$ be a composition series 
of $Z^{(\la,1)}$ as $\aA$-modules 
such that $M_i/M_{i-1}\cong L^{(\mu_i,\ve_i)}$ ($i=1,\cdots ,k$). 
We show that $L^{(\mu_i,0)}$ with $\a(\la)=\a(\mu_i)$ does not appear in the composition series of $Z^{(\la,1)}$. 
Suppose that $\a(\la)=\a(\mu_i)$ and $\ve_i=0$ for some $i$. 
Since $L^{(\mu_i,0)}\not=0$, there exists $\Ft\in \CT^+(\mu_i)$ such that 
$c_{\Ft}^{(\mu_i,0)}\not\in \rad Z^{(\mu_i,0)}$. 
Then $\Ft\in \CT(\mu_i,\nu_i)$ 
with $\a(\mu_i)=\a(\nu_i)$, 
since $\Ft \in \CT^+(\mu_i)$. 
Moreover, since $c_{\Ft}^{(\mu_i,0)}\cdot e_{\nu_i}=c_{\Ft}^{(\mu_i,0)}$,  we have  
\begin{align}
\big(c_{\Ft}^{(\mu_i,0)} + \rad Z^{(\mu_i,0)}\big)\cdot e_{\nu_i}=c_{\Ft}^{(\mu_i,0)} +\rad Z^{(\mu_i,0)}\not=0 \quad\big(\text{on }L^{(\mu_i,0)}\big).
\label{ppp}
\end{align}
Let $x_{i} \in M_i \setminus M_{i-1}$ be a representative element of $c_{\Ft}^{(\mu_i,0)} + \rad Z^{(\mu_i,0)}$ 
under the isomorphism $M_i/M_{i-1}\cong L^{(\mu_i,0)}$. 
By (\ref{ppp}), we have $x_i\cdot e_{\nu_i}\not=0$. 
On the other hand, since $x_i \in M_i \subset Z^{(\la,1)}$, we can write 
$x_i=\sum_{\Fs \in I(\la,1)} r_\Fs c_{\Fs}^{(\la,1)}$ $(r_\Fs \in R)$. 
Since $\a(\la)=\a(\mu_i)=\a(\nu_i)$, 
we have $c_{\Fs}^{(\la,1)}\cdot e_{\nu_i}=0$ for any $\Fs \in I(\la,1)$. 
Thus we have $x_i\cdot e_{\nu_i}=0$. 
This is a contradiction. 
Hence $L^{(\mu_i,0)}$ does not appear. 
This proves (\roiii). 

We can prove (\roiv) in similar way as in the proof of (\roiii), 
and (\rov) follows from (\roi) and (\roiv). 
\end{proof}

The following theorem describes the relation among various  decomposition numbers. 
Note that in the discussion below, 
we regard the decomposition numbers of $Z^{(\la,\ve)}$ (resp. $\ol{Z}^\la$) 
as zero if $(\la,\ve)\not\in \Om$ (resp. if $\la \not\in \ol{\vL}^+$).

\begin{thm}\label{decom-thm} 
Assume that $R$ is a field.  
For $\la \in \vL^+$ and $\mu \in \vL^+_0$, the followings hold. 
\begin{enumerate}
\item
If $(\mu,0)\in \Om_0$ and $\a(\la)=\a(\mu)$, then we have 
\[[W^\la:L^\mu]_\A=[Z^{(\la,0)}:L^{(\mu,0)}]_{\aA}=[\ol{Z}^{\la}:\ol{L}^{\mu}]_{\oA}.\] 
\item
If $(\mu,0)\in \Om_0$ and $\a(\la)\not=\a(\mu)$, then we have 
\[[W^\la:L^\mu]_\A=[Z^{(\la,1)}:L^{(\mu,0)}]_{\aA}.\]
\item
If $(\mu,1)\in \Om_0$, then we have 
\[[W^\la:L^\mu]_\A=[Z^{(\la,1)}:L^{(\mu,1)}]_{\aA}.\]
\end{enumerate}
\end{thm}
\begin{proof}
Take a composition series of $W^\la$ as $\A$-modules, 
and we  regard this composition series as a filtration of $W^\la$ as  $\aA$-modules by restriction. 
Combined with Proposition \ref{aA-A}, for $\ve=0,1$ such that $L^{(\mu,\ve)}\not=0$, we have 
\begin{align}
[W^\la : L^\mu]_{\A}&=[Z^{(\la,0)}\oplus Z^{(\la,1)}: L^{(\mu,\ve)}]_{\aA} 
\label{W-Z-decom} \\
&=[Z^{(\la,0)}:L^{(\mu,\ve)}]_{\aA}+[Z^{(\la,1)}:L^{(\mu,\ve)}]_{\aA}.\notag
\end{align}

First we prove (\roi). 
Let $\ve=0$, $\a(\la)=\a(\mu)$ and $(\mu,0)\in \Om_0$.  
If $(\la,0)\not\in\Om$ (if and only if $\la \not\in \ol{\vL}^+$),  
then we have $W^\la\cong Z^{(\la,1)}$ as $\aA$-modules. 
Thus, if $L^\mu$ appear in the composition series of $W^\la$ as  $\A$-modules then 
$L^{(\mu,0)}$ appear in the composition series of $Z^{(\la,1)}$ as  $\aA$-modules. 
(Note that $L^\mu \cong L^{(\mu,0)} \oplus L^{(\mu,1)}$.) 
This contradicts  Proposition \ref{decom-prop} (\roiii). 
Thus we have $[W^\la : L^\mu]_\A=0$, 
and (\roi) holds since  the decomposition numbers of $Z^{(\la,0)}$ (resp. $\ol{Z}^\la$) are zero 
if $(\la,0)\not\in \Om$ (resp. $\la \not\in \ol{\vL}^+$).  

Now suppose that $(\la,0)\in \Om$. 
Then we have $[Z^{(\la,0)}:L^{(\mu,0)}]_{\aA}=[\ol{Z}^\la : \ol{L}^\mu]_{\oA}$ 
and $[Z^{(\la,1)}:L^{(\mu,0)}]_{\aA}=0$ by Proposition \ref{decom-prop} (\roi) and (\roiii).
 (Note that we regard the decomposition numbers of $Z^{(\la,1)}$ as zero if $(\la,1)\not\in \Om$.)  
Combined with (\ref{W-Z-decom}), 
this proves (\roi). 

(\roii) and (\roiii) are proved in a similar way by combining (\ref{W-Z-decom}) with Proposition \ref{decom-prop}. 
\end{proof}

\para
Recall that, for $\mu \in \vL^+_0$, we have an isomorphism $L^\mu\cong L^{(\la,0)}\oplus L^{(\la,1)}$ as $\aA$-modules. 
This implies that  we always have 
$L^{(\la,1)}\not=0$ or $L^{(\la,0)}\not=0$. 
(It may happen that both of $L^{(\la,\ve)}$ ($\ve=0,1$) are non-zero.) 
Thus we can compute the decomposition numbers of $\A$ 
by using the decomposition numbers of $\aA$ 
thanks to Theorem \ref{decom-thm}. 

Moreover we show that, if $\A$ satisfies  the following condition (C) 
in addition to Assumption \ref{assume}, 
then 
$[W^\la :L^\mu]_\A$ for $\a(\la)=\a(\mu)$ can be always expressed in terms of $[\ol{Z}^\la :\ol{L}^\mu]_{\oA}$.

\begin{description}\vspace{1mm}
\item[(C)] $\vL\supset \vL_0^+$, and $e_\mu \not\equiv 0 \mod \A^{\vee\mu}$ for any $\mu \in \vL^+_0$. 
\end{description}

\begin{cor}
Suppose that $\A$ satisfies the condition (C). Then, for $\la \in \vL^+,\,\mu \in \vL_0^+$ such that $a(\la)=\a(\mu)$, we have 
\[[W^\la:L^\mu]_\A=[Z^{(\la,0)}:L^{(\mu,0)}]_{\aA}=[\ol{Z}^\la :\ol{L}^\mu]_{\oA}.\]
\end{cor}

In order to prove the corollary, 
we have only to show  that $L^{(\mu,0)}\not=0$ for any $\mu \in \vL^+_0$ 
under the condition (C), 
since in that case we can apply Theorem \ref{decom-thm} (\roi). 
Hence the next lemma proves the corollary. 

\begin{lem}\label{Z0-nonzero-lem}
Suppose that $\A$ satisfies the condition (C). Then, for any $\mu \in \vL_0^+$, we have $L^{(\mu,0)}\not=0$. 
\end{lem}

\begin{proof}
By Lemma \ref{e-lem}, 
we can write, 
for any $\mu \in \vL^+_0$,  
\[\dis e_\mu =\hspace{-1em}\sum_{\la\geq \mu,\, \Fs,\Ft\in \CT(\la,\mu)}\hspace{-1em}r_{\Fs\Ft}^{\la}c_{\Fs\Ft}^{\la}.\] 
Note that $c_{\Fs\Ft}^\la$ in the sum is contained in $\aA$. 
If $c_{\Fu\Fs}^\mu c_{\Ft\Fv}^\mu \equiv 0 \mod (\aA)^{\vee (\mu,0)}$ for any $\Fs,\Ft,\Fu,\Fv\in \CT(\mu,\mu)$ then 
we have $e_\mu=e_\mu e_\mu\equiv 0 \mod (\aA)^{\vee(\mu,0)}$ 
since $c_{\Fu\Fs}^\la c_{\Ft\Fv}^{\la'} \equiv 0 \mod (\aA)^{\vee(\mu,0)}$ for $\la>\mu$ or $\la'>\mu$ 
by Lemma \ref{pro-lem}. 
By using the first and the third formulas in Lemma \ref{pro-lem}, 
we have $e_\mu=e_\mu e_\mu\equiv 0 \mod \A^{\vee\mu}.$ 
This contradicts the assumption $e_\mu \not\equiv0 \mod \A^{\vee \mu}$. 
Thus, there exist $\Fs, \Ft,\Fu,\Fv\in \CT(\mu,\mu)$ such that 
$c_{\Fu\Fs}^\mu c_{\Ft\Fv}^\mu \not\equiv 0 \mod (\aA)^{\vee(\mu,0)}$. 
This implies that $\lan c_{\Fs}^{(\mu,0)},c_{\Ft}^{(\mu,0)} \ran_{\a}\not=0$, 
and we conclude that $L^{(\mu,0)}\not=0$.  
\end{proof}

\exam
\label{ex4}
We consider three algebras in Example \ref{ex} under the setting Example \ref{ex2} and Example \ref{ex3}. 
Here, we assume that $R$ is a field. 

(\roi) (Matrix algebra) 
If $R$ is a field, then $\A=M_{n\times n}(R)$ is semisimple. 
Hence, the decomposition matrix of $\A$ is the unit matrix. 
\vspace{3mm}

(\roii) (Path algebra) 
It is easy to see that $\vL_0^+=\{\la_i\,|\,1\leq i \leq 5\}\not=\vL^+$. 
Hence, $\A$ is not a quasi-hereditary algebra by \cite{GL96}. 
One can calculate the following decomposition matrix of $\A$ directry. 
{\small\[ \begin{array}{c|ccccc} 
&\la_1&\la_2&\la_3&\la_4&\la_5 \\\hline 
\la_0 &1&0&0&0&0 \\
\la_1 &1&1&0&0&0 \\
\la_2 &0&1&1&0&0 \\
\la_3 &0&0&1&1&0 \\
\la_4 &0&0&0&1&1 \\
\la_5 &0&0&0&0&1 \\
\end{array} \quad 
\]}
where $(i,j)$-entry denotes  $[W^{\la_{i-1}}:L^{\la_j}]_\A$.  
Moreover, $\Om_0=\{(\la_i,0)\,|\,1\leq i \leq 5\}$, $\ol{\vL}^+_0=\ol{\vL}^+$ 
and we have the following decomposition matrices of $\aA$ and $\oA$. 
{\small
\[\begin{array}{c|ccccc} 
&(\la_1,0)&(\la_2,0)&(\la_3,0)&(\la_4,0)&(\la_5,0) \\\hline 
(\la_0,1) &1&0&0&0&0 \\
(\la_1,0) &1&1&0&0&0 \\
(\la_2,0) &0&1&1&0&0 \\
(\la_3,0) &0&0&1&0&0 \\
(\la_3,1) &0&0&0&1&0 \\
(\la_4,0) &0&0&0&1&1 \\
(\la_5,0) &0&0&0&0&1 \\
\end{array} \hspace{3em}
\begin{array}{c|ccccc} 
&\la_1&\la_2&\la_3&\la_4&\la_5 \\\hline 
\la_1 &1&1&0&0&0 \\
\la_2 &0&1&1&0&0 \\
\la_3 &0&0&1&0&0 \\
\la_4 &0&0&0&1&1 \\
\la_5 &0&0&0&0&1 \\
\end{array}
\]}
\vspace{-1mm}

\hspace{6em}(Matrix of $\aA$)
\hspace{12em}(Matrix of $\oA$)\vspace{3mm}

The left upper $4\times3$ submatrix (resp. right lower $3\times 2$ submatrix) of the above decomposition matrix of $\aA$ 
gives the decomposition matrix of $(RQ_1^\a/\CI_1^\a)$ (resp. $(RQ_2^\a/\CI_2^\a)$). 
Similarly, 
the left upper $3\times3$ submatrix (resp. right lower $2\times 2$ submatrix) of the above decomposition matrix of $\oA$ 
gives the decomposition matrix of $(R\ol{Q}_1^\a/\ol{\CI}_1^\a)$ (resp. $(R\ol{Q}_2^\a/\ol{\CI}_2^\a)$). 
In particular, $R\ol{Q}_i^\a/\ol{\CI}_1^\a$ ($i=1,2$) is a quasi-hereditary algebra. 
\vspace{3mm}

(\roiii) (Cyclotomic $q$-Schur algebra) 
We know that $\ol{\vL}^+=\vL^+$ by \cite{SW}. 
Moreover, in \cite[Corollary 4.16]{SW}, 
it is shown that, for $\Bla, \Bmu \in \vL^+$, 
$\ol{Z}^{\Bla} \cong W^{\Bla^{[1]}}\otimes \cdots \otimes W^{\Bla^{[g]}}$ and 
$\ol{L}^{\Bmu}\cong L^{\Bmu^{[1]}}\otimes \cdots \otimes L^{\Bmu^{[g]}}$ under the isomorphism (\ref{decom-oSp}), 
where $W^{\Bla^{[k]}}$ is the standard module of $\Sc(\vL_{n_k,r_k})$ with $n_k=|\Bla^{[k]}|$ and 
$L^{\Bmu^{[k]}}$ is the simple module of $\Sc(\vL_{n_k,r_k})$ with $n_k=|\Bmu^{[k]}|$. 
Combining this with Proposition \ref{decom-prop} and Theorem \ref{decom-thm}, 
we have the following product formula (\cite[Theorem 4.17]{SW}), for $\Bla,\Bmu \in \vL^+$ such that $\a_{\Bp}(\Bla)=\a_{\Bp}(\Bmu)$, 
\[[W^{\Bla} : L^{\Bmu}]_{\Sc} 
=[\ol{Z}^{\Bla}:\ol{L}^{\Bmu}]_{\oSp}
=\prod_{k=1}^g [W^{\Bla^{[k]}}:L^{\Bmu^{[k]}}]_{\Sc(\vL_{n_k,r_k})}.\]

\section{Parabolic type subalgebras}\label{tA}
In this section, we introduce certain parabolic type subalgebras $\tA$ and  $(\tA)^\ast$ of $\A$, 
and study the relations among $\A$, $\aA$ and $\oA$. 

\para
For $(\la,\ve)\in \Om$, set 
\[\wt{I}(\la,\ve)=\begin{cases} \CT^+(\la) &\text{if }\ve=0\\ \CT^-(\la) &\text{if }\ve=1, \end{cases}\qquad 
	\wt{J}(\la,\ve)=\begin{cases} \CT^+(\la) &\text{if }\ve=0\\ \CT(\la) &\text{if }\ve=1, \end{cases}\]
and  
\[\wt{\ZC}^\a(\la,\ve)=\{c_{\Fs\Ft}^\la\,|\, (\Fs,\Ft)\in \wt{I}(\la,\ve)\times \wt{J}(\la,\ve)\}.\] 
For $\la \in \vL^+$, put 
$\wt{\ZC}^\a(\la)=\wt{\ZC}^\a(\la,0) \cup \wt{\ZC}^\a(\la,1)$. 
We define a subset $\wt{\ZC}^\a$ of $\ZC$ by 
\begin{align*}
\wt{\ZC}^\a&=\big\{c_{\Fs\Ft}^\la \bigm| (\Fs,\Ft)\in \wt{I}(\la,\ve)\times \wt{J}(\la,\ve) \quad \text{for }(\la,\ve)\in \Om\big\}\\
&=\coprod_{(\la,\ve)\in \Om}\wt{\ZC}^\a(\la,\ve).
\end{align*}

We define an $R$-submodule $\tA$ of $\A$ as the $R$-span of  $\wt{\ZC}^\a$. 
We show that $\tA$ is a subalgebra of $\A$ in a similar way as in the case of $\aA$.   

\begin{lem}\label{pro-lem-para}
For $c_{\Fs_1\Ft_1}^{\la_1}, c_{\Fs_2\Ft_2}^{\la_2} \in \wt{\ZC}^\a$ 
such that  $\Fs_i \in \CT(\la_i,\mu_i)$ and $\Ft_i\in \CT(\la_i,\nu_i)$ ($i=1,2$), 
we have 
\[ c_{\Fs_1\Ft_1}^{\la_1}c_{\Fs_2\Ft_2}^{\la_2} =\begin{cases}
	\dis \sum_{c_{\Fs,\Ft}^{\la_1}\in \wt{\ZC}^\a(\la_1,0)}r_{\Fs\Ft}^{\la_1}c_{\Fs\Ft}^{\la_1} 
		+ \sum_{\la>\la_1}\sum_{c_{\Fs,\Ft}^{\la}\in \wt{\ZC}^\a(\la)}r_{\Fs\Ft}^{\la}c_{\Fs\Ft}^{\la}, 
		&\text{if } c_{\Fs_1 \Ft_1}^{\la_1}\in \wt{\ZC}^\a(\la_1,0), \\\,\\
	\dis \sum_{\la\geqq\la_1}\sum_{c_{\Fs,\Ft}^{\la}\in \wt{\ZC}^\a(\la,1)}r_{\Fs\Ft}^{\la}c_{\Fs\Ft}^{\la}, 
		&\text{if } c_{\Fs_1 \Ft_1}^{\la_1}\in \wt{\ZC}^\a(\la_1,1),\\\,\\
	\dis \sum_{\la\geqq\la_2}\sum_{c_{\Fs,\Ft}^{\la}\in \wt{\ZC}^\a(\la)}r_{\Fs\Ft}^{\la}c_{\Fs\Ft}^{\la}, 
		&\text{if } c_{\Fs_2 \Ft_2}^{\la_2}\in \wt{\ZC}^\a(\la_2,0),\\\,\\
	\dis \sum_{\la\geqq\la_2}\sum_{c_{\Fs,\Ft}^{\la}\in \wt{\ZC}^\a(\la,1)}r_{\Fs\Ft}^{\la}c_{\Fs\Ft}^{\la},
		&\text{if } c_{\Fs_2 \Ft_2}^{\la_2}\in \wt{\ZC}^\a(\la_2,1).
\end{cases}\]
\end{lem}
\begin{proof}
This lemma is proved in  a similar way as in the proof of Lemma \ref{pro-lem}. 
Note that the third formula is weaker than the corresponding formula in Lemma \ref{pro-lem} 
since it may occur $\la_2>\la_1$ with $\a(\la_2)=\a(\la_1)$.  
\end{proof}
This lemma implies that $\tA$ is a subalgebra of $\A$. 
Moreover, $\tA$ turns out to be a standardly based algebra in the sence of \cite{DR98},  
namely the following theorem holds.   
The proof is  similar to Theorem \ref{aA-thm}. 

\begin{thm} \label{tA-thm}
$\tA$ is a subalgebra of $\A$ containing $\aA$. Moreover, 
$\tA$ turns out to be a standardly based algebra with a standard basis $\wt{\ZC}^\a$, 
i.e., the following holds; for any $a \in \tA, c_{\Fs\Ft}^\la \in \wt{\ZC}^\a(\la,\ve)$, we have 
\begin{align}
a \cdot c_{\Fs\Ft}^\la \equiv \sum_{\Fu\in \wt{I}(\la,\ve) }r_{\Fu}^{(a,\Fs)} c_{\Fu\Ft}^{\la} 
	\mod (\tA)^{\vee (\la,\ve)}\qquad (r_{\Fu}^{(a,\Fs)}\in R), \label{aaap}\\*
c_{\Fs\Ft}^\la \cdot a \equiv \sum_{\Fv\in \wt{J}(\la,\ve) }r_{\Fv}^{(\Ft,a)} c_{\Fs\Fv}^{\la} 
\mod (\tA)^{\vee (\la,\ve)}\qquad (r_{\Fv}^{(\Ft,a)}\in R),  \label{bbbp}
\end{align}
where $(\tA)^{\vee(\la,\ve)}$ is an $R$-submodule of $\tA$ spanned by 
$\big\{c_{\Fs'\Ft'}^{\la'}\in \wt{\ZC}^\a(\la',\ve')\,|\,(\la',\ve')>(\la,\ve)\big\}$, 
and $r_{\Fu}^{(a,\Fs)}$ (resp. $r_{\Fv}^{(\Ft,a)}$) does not depend on the choice of $\Ft\in \wt{J}(\la,\ve)$ (resp. $\Fs\in \wt{I}(\la,\ve)$). 
\end{thm}

\para \label{tA-general}
We apply the general theory of standardly based algebras (\cite{DR98}) to $\tA$. 
For $(\la,\ve)\in \Om$, 
let $\wt{Z}^{(\la,\ve)}$ be a standard right $\tA$-module, 
namely $\wt{Z}^{(\la,\ve)}$ has an $R$-free basis 
$\{\wt{c}_\Ft^{(\la,\ve)}\,|\, \Ft \in \wt{J}(\la,\ve)\}$, 
and the (right) action of $\tA$ on $\wt{Z}^{(\la,\ve)}$ is defined by 
\[\wt{c}_{\Ft}^{(\la,\ve)}\cdot a =\sum_{\Fv\in \wt{J}(\la,\ve)}r_{\Fv}^{(\Ft,a)}\wt{c}_{\Fv}^{(\la,\ve)} \qquad(\Ft\in \wt{J}(\la,\ve),\,a \in \tA),\]
where $r_{\Fv}^{(\Ft,a)}\in R$ is a coefficient appeared in (\ref{bbbp}). 

Similarly, we define a standard left $\tA$-module  $\sh\wt{Z}^{(\la,\ve)}$
with an $R$-free basis $\{\,\sh\wt{c}_\Fs^{(\la,\ve)}\,|\, \Fs \in \wt{I}(\la,\ve)\}$ 
by the (left) action of $\tA$ given by 
\[a\cdot\,\sh\wt{c}_{\Fs}^{(\la,\ve)} =\sum_{\Fu\in \wt{I}(\la,\ve)}r_{\Fu}^{(a,\Fs)}\,\sh\wt{c}_{\Fu}^{(\la,\ve)} 
	\qquad(\Fs\in \wt{I}(\la,\ve),\,a \in \tA),\]
where $r_{\Fu}^{(a,\Fs)}\in R$ is given in (\ref{aaap}). 
We can define a bilinear form $\b : \,\sh \wt{Z}^{(\la,\ve)}\times \wt{Z}^{(\la,\ve)}\ra R$ by  
\[\b(\,\sh \wt{c}_{\Fs}^{(\la,\ve)}, \wt{c}_{\Ft}^{(\la,\ve)})c_{\Fu\Fv}^\la \equiv c_{\Fu\Ft}^{\la}c_{\Fs\Fv}^{\la} \mod (\tA)^{\vee (\la,\ve)} 
	\qquad(\Fs, \Fu \in \wt{I}(\la,\ve),\,\Ft,\Fv \in \wt{J}(\la,\ve)).\]
Note that this definition does not depend on the choice of $\Fu \in \wt{I}(\la,\ve)$ and $\Fv \in \wt{J}(\la,\ve)$.  
Put
\begin{align*}
&\rad \wt{Z}^{(\la,\ve)}=\big\{x \in \wt{Z}^{(\la,\ve)}\bigm| \b(\,\sh y,x)=0 \text{ for any }\,\sh y \in \,\sh \wt{Z}^{(\la,\ve)}\big\},\\* 
&\rad \,\sh \wt{Z}^{(\la,\ve)}=\big\{\,\sh x \in \,\sh \wt{Z}^{(\la,\ve)}\bigm| \b(\,\sh x,y)=0 \text{ for any }y \in \, \wt{Z}^{(\la,\ve)}\big\}.
\end{align*} 
Then $\rad \wt{Z}^{(\la,\ve)}$ (resp. $\rad \,\sh \wt{Z}^{(\la,\ve)}$) turns out to be 
an $\tA$-submodule of $\wt{Z}^{(\la,\ve)}$ (resp. of $\sh \wt{Z}^{(\la,\ve)}$).   
Let $\wt{L}^{(\la,\ve)}=\wt{Z}^{(\la,\ve)}/\rad \wt{Z}^{(\la,\ve)}$ 
(resp. $\sh \wt{L}^{(\la,\ve)}=\,\sh \wt{Z}^{(\la,\ve)}/\rad\,\sh \wt{Z}^{(\la,\ve)}$) 
be the quotient $\tA$-module. 
Set $\wt{\Om}_0=\{(\la,\ve)\in \Om\,|\, \wt{L}^{(\la,\ve)}\not=0\}
=\{(\la,\ve)\in \Om\,|\, \,\sh\wt{L}^{(\la,\ve)}\not=0\}$. 
We have the following corollary by \cite[Theorem (2.4.1)]{DR98}.  

\begin{cor}
Assume that $R$ is a field. Then  $\{\wt{L}^{(\la,\ve)}\,|\, (\la,\ve) \in \wt{\Om}_0\}$ 
$\big($resp. $\{\sh \wt{L}^{(\la,\ve)}\,|\\ (\la,\ve) \in \wt{\Om}_0\}$$\big)$ 
is a complete set of non-isomorphic right (resp. left) simple $\tA$-modules. 
\end{cor} 

\para
We define  $(\tA)^*=\{a^*\,|\,a\in \tA\}$, where $*$ is the anti-automorphism on $\A$ 
given in \ref{def-cellular}. 
Then one can check that 
$(\tA)^*$ is a standardly based algebra with a standard basis 
$(\wt{\ZC}^\a)^*=\big\{c_{\Ft\Fs}^\la= (c_{\Fs\Ft}^\la)^* \bigm| c_{\Fs\Ft}^\la \in \wt{\ZC}^\a \big\}$ 
in a similar way as in the case of $\tA$. 

\begin{prop}\
\begin{enumerate}
\item
$\tA\cap (\tA)^*= \aA. $
\item
$\A=\tA \cdot (\tA)^\ast =(\tA)^\ast \cdot \tA$. 
\end{enumerate}
\end{prop}

\begin{proof}
(\roi) is  clear by comparing the cellular bases of $\tA$, $(\tA)^\ast$ and of $\aA$. 

For (\roii), we  only show the first equality, as the second equality is obtained in a similar way. 
It is clear that $\A \supset \tA \cdot (\tA)^\ast$. 
Thus it is enough to show that $\A \subset \tA \cdot (\tA)^\ast$. 
For $c_{\Fs\Ft}^\la \in \ZC$, we have $c_{\Fs\Ft}^\la \in \tA$ or $c_{\Fs\Ft}^\la \in (\tA)^\ast$. 
(Note that it may happen that  $c_{\Fs\Ft}^\la \in \tA$ and $c_{\Fs\Ft}^\la \in (\tA)^\ast$.)
Moreover, we know that $1_\A \in \tA$ and $1_\A \in (\tA)^\ast$. 
Thus, we have $c_{\Fs\Ft}^\la \cdot 1_{\A}\in \tA \cdot (\tA)^\ast$ if $c_{\Fs\Ft}^\la \in \tA$, 
and $1_\A \cdot c_{\Fs\Ft}^\la \in \tA \cdot (\tA)^\ast$ if $c_{\Fs\Ft}^\la\in (\tA)^\ast$. 
This implies that $\A \subset \tA\cdot (\tA)^\ast$. 
\end{proof}

\para \label{bl}
Recall that $\oA$ is the quotient algbera of $\aA$. 
We can also interpret $\oA$ as a quotient algebra of $\tA$ in the following way. 
Recall that $\wh{\Om}=\{(\la,\ve)\in \Om\,|\,\ve=1\}$. 
Let $\tA(\wh{\Om})$ be an $R$-submodule of $\tA$ spanned by 
$\big\{c_{\Fs\Ft}^\la \bigm| (\Fs,\Ft)\in \wt{I}(\la,1)\times \wt{J}(\la,1) \text{ for some } (\la,1)\in \wh{\Om}\big\}$. 
By the second and the fourth formulas in Lemma \ref{pro-lem-para}, 
$\tA(\wh{\Om})$  turns out to be a two-sided ideal of $\tA$. 
Thus we can define a quotient algebra $\tA/\tA(\wh{\Om})$, 
which is clearly isomorphic to  $\oA$.  
(Note that $\tA(\wh{\Om}) \cap \aA= \aA(\wh{\Om})$.) 
Thus we have the following commutative diagram.\\ 
\hspace{5em}
\begin{picture}(30,60)
\put(0,40){$\aA$ \scalebox{2}[1]{$\hookrightarrow $} $\tA$ \scalebox{2}[1]{$\hookrightarrow $} $\A$}
\put(50,0){$\oA$}
\put(5,10){\rotatebox{140}{\scalebox{4}[0.6]{$\twoheadleftarrow$}}}
\put(53,12){\rotatebox{90}{\scalebox{2.2}[1.2]{$\twoheadleftarrow$}}}
\end{picture}\vspace{2mm}\\

\para
Recall that, for $\la \in \vL^+$, $\CT^+(\la)=I(\la,0)=\wt{I}(\la,0)=\wt{J}(\la,0)$, 
$\CT^-(\la)=I(\la,1)=\wt{I}(\la,1)$ and $\CT(\la)=\wt{J}(\la,1)$. 
The following relations among bilinear forms 
$\lan\,,\,\ran$, $\lan\,,\,\ran_\a$, $\lan\,,\,\ran_{\ol{\a}}$ and $\b(\,,\,)$ 
are immediate from the definitions. 
\begin{align}
&\lan c_{\Fs}^\la ,c_{\Ft}^\la \ran 
	=\b( \,\sh\wt{c}_{\Ft}^{(\la,0)}, \wt{c}_{\Fs}^{(\la,0)}) 
	=\lan c_{\Fs}^{(\la,0)},c_{\Ft}^{(\la,0)}\ran_\a 
	=\lan \ol{c}_{\Fs}^\la, \ol{c}_{\Ft}^\la \ran_{\ol{\a}}
\quad \text{for } \Fs,\Ft \in \CT^+(\la). \label{bl1}\\ 
&\lan c_{\Fs}^\la ,c_{\Ft}^\la \ran 
	=\b( \,\sh\wt{c}_{\Ft}^{(\la,1)}, \wt{c}_{\Fs}^{(\la,1)}) 
	=\lan c_{\Fs}^{(\la,1)},c_{\Ft}^{(\la,1)}\ran_\a 
\quad \text{for }\Fs,\Ft \in \CT^-(\la). \label{bl2}\\ 
&\lan c_{\Fs}^\la ,c_{\Ft}^\la \ran =0 
	=\b( \,\sh\wt{c}_{\Ft}^{(\la,1)}, \wt{c}_{\Fs}^{(\la,1)})
\quad \text{for }\Fs \in \CT^+(\la),\Ft \in \CT^-(\la). \label{bl3}\\
&\lan c_{\Fs}^\la ,c_{\Ft}^\la \ran =0 
\quad \text{for }\Fs \in \CT^-(\la),\Ft \in \CT^+(\la). \label{bl4}
\end{align}

\para 
For a cellular algebra $\A$, 
a right standard module $W^\la$ and a left standard module $\sh W^\la$ have symmetric properties 
thanks to the algebra anti-automorphism\nobreak\hspace{1mm}$\ast$. (See Section \ref{cellular}.) 
Thus, we have only to consider right standard modules for $\A$. 
For $\aA$ and $\oA$, the situation  is similar. 
But, since $\tA$ does not have such algebra anti-automorphism $\ast$, 
we need to consider left and right standard modules for $\tA$. 
For this we prepare some notation for left standard modules of $\aA$. 
(For left standard modules of $\A$, we follow the notation in Section \ref{cellular}).  
For $(\la,\ve)\in \Om$, let $\sh Z^{(\la,\ve)}$ be a left standard module of $\aA$ 
with an $R$-free basis $\{\,\sh c_{\Ft}^{(\la,\ve)}\,|\,\Ft \in I(\la,\ve)\}$.  
Put $\sh L^{(\la,\ve)}=\,\sh Z^{(\la,\ve)}/ \rad\,\sh Z^{(\la,\ve)}$. 
Then $\{\,\sh L^{(\la,\ve)}\,|\,(\la,\ve)\in \Om_0\}$ is a complete set of non-isomorphic 
left simple modules of $\aA$ when $R$ is a field. 

If we regard $\tA$-modules as $\aA$-modules by restriction, 
we have the following results. 

\begin{prop}[right module structure for $\tA$ and $\aA$]\ \label{tA-aA}

For $\la \in \vL^+$, the following holds.
\begin{enumerate}
\item
There exists an isomorphism ${Z}^{(\la,0)} \stackrel{\sim}{\ra} \wt{Z}^{(\la,0)}$ of $\aA$-modules by the assignment  
${c}_{\Ft}^{(\la,0)}\mapsto \wt{c}_{\Ft}^{(\la,0)} \,\,\big(\Ft\in \CT^+(\la)\big)$.
\item
There exists an isomorphism ${Z}^{(\la,0)} \oplus {Z}^{(\la,1)} \stackrel{\sim}{\ra} \wt{Z}^{(\la,1)}$ of $\aA$-modules by the assignment 
$({c}_{\Ft}^{(\la,0)},0) \mapsto \wt{c}_{\Ft}^{(\la,1)}\,\,\big(\Ft\in \CT^+(\la)\big)$ and 
$(0,{c}_{\Ft}^{(\la,1)}) \mapsto \wt{c}_{\Ft}^{(\la,1)}\,\,\big(\Ft\in \CT^-(\la)\big)$. 
\item
${L}^{(\la,0)} \cong \wt{L}^{(\la,0)}$ as $\aA$-modules.
\item
${L}^{(\la,1)} \cong \wt{L}^{(\la,1)}$ as $\aA$-modules.
\end{enumerate}
\end{prop}
\begin{proof}
Note that 
we can replace $(\aA)^{\vee(\la,\ve)}$ by $(\aA\cap (\tA)^{\vee(\la,\ve)})$ 
in (\ref{cellular-rel-aA}), 
then 
(\roi) and (\roii) are obtained in similar way as in the proof of Proposition \ref{aA-A}. 
By (\ref{bl1}), we have $\rad Z^{(\la,0)} = \rad \wt{Z}^{(\la,0)}$ under the isomorphism $Z^{(\la,0)}\cong \wt{Z}^{(\la,0)}$ as $\aA$-modules. 
This implies (\roiii). 
Moreover, by (\ref{bl1}) $\sim$ (\ref{bl3}), we have $Z^{(\la,0)} \oplus \rad Z^{(\la,1)} =\rad \wt{Z}^{(\la,1)}$ under the isomorphism 
${Z}^{(\la,0)} \oplus {Z}^{(\la,1)} \cong \wt{Z}^{(\la,1)}$ as $\aA$-modules. 
This proves (\roiv). 
\end{proof}

\begin{prop}[left module structure for  $\tA$ and $\aA$]\ \label{tA-aA-left}

For $\la \in \vL^+$, the following holds.
\begin{enumerate}
\item
There exists an isomorphism $\sh{Z}^{(\la,0)} \stackrel{\sim}{\ra} \,\sh \wt{Z}^{(\la,0)}$ of $\aA$-modules by the assignment 
$\sh{c}_{\Ft}^{(\la,0)}\mapsto \,\sh\wt{c}_{\Ft}^{(\la,0)} \,\,\big(\Ft\in \CT^+(\la)\big)$.
\item
There exists an isomorphism $\sh{Z}^{(\la,1)} \stackrel{\sim}{\ra} \,\sh \wt{Z}^{(\la,1)}$ of $\aA$-modules by the assignment 
$\sh{c}_{\Ft}^{(\la,1)}\mapsto \sh\wt{c}_{\Ft}^{(\la,1)} \,\,\big(\Ft\in \CT^-(\la)\big)$.
\item
$\sh{L}^{(\la,0)} \cong \,\sh\wt{L}^{(\la,0)}$ as $\aA$-module.
\item
$\sh{L}^{(\la,1)} \cong \,\sh\wt{L}^{(\la,1)}$ as $\aA$-module.
\end{enumerate}
\end{prop}
\begin{proof}
(\roi) and (\roiii) are similar to the proof of Proposition \ref{tA-aA}.  
We remark that $\sh\wt{Z}^{(\la,1)}$ has a basis indexed by $\wt{I}(\la,1)$. 
Then (\roii) and (\roiv) are obtained in a similar way as (\roi), (\roiii), 
which is simpler than Proposition \ref{tA-aA}. 
\end{proof}

If we regard $\A$-modules as $\tA$-modules by restriction,  
we have the following results. 
\begin{prop}[right module structure for $\tA$ and $\A$]\ \label{tA-A}

For $\la \in \vL^+$, the following holds. 
\begin{enumerate}
\item
There exists an injective  homomorphism $\wt{f}_\la:\wt{Z}^{(\la,0)}\ra W^\la$ of $\tA$-modules by the assignment 
$\wt{c}_{\Ft}^{(\la,0)} \mapsto {c}_{\Ft}^{\la}$ $\big(\Ft\in \CT^+(\la)\big)$. 
\item
There exists an isomorphism $\wt{g}_\la:\wt{Z}^{(\la,1)} \stackrel{\sim}{\ra} W^\la$ of $\tA$-modules by the assignment 
$\wt{c}_{\Ft}^{(\la,1)} \mapsto c_{\Ft}^\la$ $\big(\Ft\in \CT(\la)\big)$. 
\item
$L^\la$ contains $\wt{L}^{(\la,0)}$ as an $\tA$-submodule. 
\item
$L^\la/\wt{L}^{(\la,0)}\cong \wt{L}^{(\la,1)}$ as $\tA$-modules. 
\end{enumerate}
\end{prop} 
\begin{proof}
Note that, in (\ref{bbbp}), 
we can replace $(\tA)^{\vee (\la,\ve)}$ by $(\tA \cap \A^{\vee \la})$ 
thanks to the first and the second formulas in Lemma \ref{pro-lem-para}, 
then
(\roi) and (\roii) are obtained in a similar way as in the proof of Proposition \ref{aA-A} and Proposition \ref{tA-aA}. 
By (\ref{bl1}) and (\ref{bl3}), 
we see that $\wt{f}_\la^{-1}(\rad W^\la) = \rad \wt{Z}^{(\la,0)}$. 
Thus $\wt{f}_\la$ induces the injective $\tA$-homomorphism $\wt{L}^{(\la,0)}\ra L^\la$. 
This proves (\roiii). 
By (\ref{bl2}) and (\ref{bl3}), 
we see that $\wt{g}^{-1}_\la(\rad W^\la) \subset \rad \wt{Z}^{(\la,1)}$. 
Thus $\wt{g}^{-1}_\la$ induces a surjective $\tA$-homomorphism $L^\la \ra \wt{L}^{(\la,1)}$. 
This implies that there exists an $\tA$-submodule $N^\la$ of $L^\la$ such that $L^\la/N^\la \cong \wt{L}^{(\la,1)}$. 
But we have $N^\la \cong {L}^{(\la,0)}$ as $\aA$-modules by Proposition \ref{aA-A} and Proposition \ref{tA-aA}. 
Since $\wt{L}^{(\la,0)}\subset L^\la$, 
we must have $N^\la = \wt{L}^{(\la,0)}$. 
This proves (\roiv). 
\end{proof}

\begin{prop}[left module structure for $\tA$ and $\A$]\ \label{tA-A-left}

For $\la \in \vL^+$, the following holds. 
\begin{enumerate}
\item
There exists an injective  homomorphism $\sh\wt{f}_\la:\,\sh\wt{Z}^{(\la,1)}\ra \,\sh W^\la$ of $\tA$-modules by the assignment 
$\sh\wt{c}_{\Ft}^{(\la,1)} \mapsto \,\sh{c}_{\Ft}^{\la}$ $\big(\Ft\in \CT^-(\la)\big)$. 
\item
$\sh L^\la$ contains $\sh\wt{L}^{(\la,1)}$ as an $\tA$-submodule. 
\item
$\sh L^\la/\,\sh\wt{L}^{(\la,1)}\cong \,\sh\wt{L}^{(\la,0)}$ as $\tA$-modules. 
\end{enumerate}
\end{prop}
\begin{proof}
Note that, in (\ref{aaap}), 
we can replace $(\tA)^{\vee (\la,1)}$ by $(\tA \cap \A^{\vee \la})$ 
thanks to the fourth formula in Lemma \ref{pro-lem-para}, 
then 
(\roi) and (\roii) are similar to the proof of Proposition \ref{tA-A}. 
We prove (\roiii). 
One can define a surjective $\tA$-homomorphism $\sh\wt{g}:\,\sh W^\la \ra \,\sh\wt{Z}^{(\la,0)}$ 
by the assinment $\sh c_{\Ft}^\la \mapsto \,\sh c_{\Ft}^{(\la,0)}$ for $\Ft \in \wt{I}(\la,0)$ 
and $\sh c_{\Ft}^\la \mapsto \,0$ for $\Ft \in \wt{I}(\la,1)$. 
Let $\sh\wt{g}' : \,\sh W^\la \ra \,\sh\wt{L}^{(\la,0)}$ be a surjective $\tA$-homomorphism 
obtained by the composition of $\sh\wt{g}$ and the natural surjection $\sh\wt{Z}^{(\la,0)} \ra \,\wt{L}^{(\la,0)}$. 
Since $\rad \,\sh W^\la \subset \ker\,\sh\wt{g}'$, 
$\sh\wt{g}'$ induces a surjective $\tA$-homomorphism $\sh L^\la \ra \,\sh \wt{L}^{(\la,0)}$. 
Hence, there exists an $\tA$-submodule $\sh N^\la$ of $\sh L^\la$ 
such that $\sh L^\la / \,\sh N^\la \cong \,\sh\wt{L}^{(\la,0)}$, 
and we must have $\sh N^\la \cong \,\sh\wt{L}^{(\la,1)}$ in similar way as in the proof of Proposition \ref{tA-A}. 
\end{proof}
\remarks

(\roi) 
$\tA$ is  not semisimple except the case where $\A=\tA=\aA$, 
which is shown as follows.  
For any $\la \in \vL^+$, 
$\wt{Z}^{(\la,0)}\subset \rad \wt{Z}^{(\la,1)}$ always holds 
if we regard $\wt{Z}^{(\la,0)}$ as an $\tA$-submodule of $\wt{Z}^{(\la,1)}$ by the injection $\wt{g}^{-1}_\la \circ  \wt{f}_\la$. 
If $\tA$ is semisimple then we have $\rad \wt{Z}^{(\la,1)}=0$ for any $\la \in \vL^+$ 
by the general theory of standardly based algebras (\cite{DR98}). 
This implies that $\wt{Z}^{(\la,0)}=0$ for any $\la \in \vL^+$. 
In this case, 
$\CT(\la)=\CT^-(\la)$ for $\la \in \vL^+$, 
and so we have  $\A=\tA=\aA$. 

(\roii) 
Contrast to the case of $\aA$, 
a  formula as in Proposition \ref{aA-A} (\roii) 
does not hold for $\tA$. 
In fact, 
suppose the following property holds for $\tA$. \vspace{-3mm}\\
\begin{description}
 \item[($\ast$)] For any $\la \in \vL^+$, $L^\la \cong \wt{L}^{(\la,0)}\oplus \wt{L}^{(\la,1)}$  as $\tA$-modules.
\end{description}\vspace{1mm}
Assume that $R$ is a field, $\A$ is semisimple, and that $\A\not=\aA$. 
We regard $\A$ as a right $\A$-module by multiplications. 
Then $\A$ is  semisimple as a right $\A$-module. 
By $(\ast)$, $\A$ is also semisimple as a right $\tA$-module. 
Thus $\tA$-submodule $\tA$ of $\A$ is also semisimple as a right $\tA$-module. 
This is equivalent with that $\tA$ is semisimple as an $R$-algebra, 
which  contradicts  (\roi). 

(\roiii) 
By Proposition \ref{tA-A} (\roi) and (\roii), for $\la \in \vL^+$, we have the following exact sequence of right $\tA$-modules. 
\begin{align*}
0 \ra \wt{Z}^{(\la,0)} \ra \wt{Z}^{(\la,1)} \ra \wt{Z}^{(\la,1)}/\wt{Z}^{(\la,0)} \ra 0 .\label{exact-tA}
\end{align*}
This exact sequence does not split in general 
since  the splitness of this exact sequence for any $\la \in \vL^+$ implies 
the property ($\ast$) in (\roii) 
by a similar argument as in the proof of Proposition \ref{aA-A} (\roii). 
However we note that 
this exact sequence, regarded as a sequence of $\aA$-modules by restriction, 
splits by Proposition \ref{tA-aA}.

(\roiv) 
In (\ref{aaap}), 
we can not replace $(\tA)^{\vee (\la,0)}$ by $(\tA \cap \A^{\vee \la})$\,
(see the third formula in Lemma \ref{pro-lem-para}). 
Hence, 
for $(\la,0)\in \Om$, $\sh \wt{Z}^{(\la,0)}$ is not an $\tA$-submodule of $\sh W^\la$, and 
$\sh \wt{L}^{(\la,0)}$ is not an $\tA$-submodule of $\sh L^\la$ in general.

\para
In the rest of this section, we assume that $R$ is a field. 
We consider relations between decomposition numbers  $[\wt{Z}^{(\la,\ve)}:\wt{L}^{(\mu,\ve')}]_{\tA}$ 
$\big((\la,\ve)\in \Om,\,(\mu,\ve')\in \Om_0\big)$ of $\tA$ 
and decomposition numbers of $\A,\aA,\oA$ appeared in the previous section. 
By Theorem \ref{decom-thm} and Proposition \ref{tA-aA}, 
we have the following theorem. 
\begin{thm} \label{decom-thm-tA}
For $\la \in \vL^+$ and $\mu \in \vL^+_0$, we have followings. 
\begin{enumerate}
\item
If $(\mu,0)\in \Om_0$ and $\a(\la)=\a(\mu)$ then we have 
\[[W^\la:L^\mu]_\A=[Z^{(\la,0)}:L^{(\mu,0)}]_{\aA}=[\wt{Z}^{(\la,0)}:\wt{L}^{(\mu,0)}]_{\tA}=[\ol{Z}^{\la}:\ol{L}^{\mu}]_{\oA}.\] 
\item
If $(\mu,0)\in \Om_0$ and $\a(\la)\not=\a(\mu)$ then we have 
\[[W^\la:L^\mu]_\A=[Z^{(\la,1)}:L^{(\mu,0)}]_{\aA}=[\wt{Z}^{(\la,1)}:\wt{L}^{(\mu,0)}]_{\tA}.\]
\item
If $(\mu,1)\in \Om_0$ then we have 
\[[W^\la:L^\mu]_\A=[Z^{(\la,1)}:L^{(\mu,1)}]_{\aA}=[\wt{Z}^{(\la,1)}:\wt{L}^{(\mu,1)}]_{\tA}.\]
\end{enumerate}
\end{thm}
\begin{proof}
In view of Theorem \ref{decom-thm}, 
we have only to compare the decomposition numbers of $\aA$ and $\tA$. 
For (\roi),  
take a composition series of $\wt{Z}^{(\la,0)}$ as $\tA$-modules. 
We regard this composition series as a filtration of $\aA$-modules by restriction. 
Then this filtration coincides with a composition series of $Z^{(\la,0)}$ as $\aA$-modules by Proposition \ref{tA-aA}. 
This implies (\roi). 

For (\roii),  
similarly, a composition series of $\wt{Z}^{(\la,1)}/\wt{Z}^{(\la,0)}$ as $\tA$-modules coincides with 
a composition series of $Z^{(\la,1)}$ as $\aA$-modules by Proposition \ref{tA-aA}. 
Moreover, $\wt{L}^{(\mu,0)}$ such that $\a(\la)\not=\a(\mu)$ 
does not appear in the composition series of $\wt{Z}^{(\la,0)}$ as $\tA$-modules 
since $L^{(\mu,0)}$ such that $\a(\la)\not=\a(\mu)$ does not appear in a composition series of $Z^{(\la,0)}$ (see Proposition \ref{decom-prop}). 
This implies (\roii). 
(\roiii) is proved in a similar way as in (\roii). 
\end{proof}
Combining Theorem \ref{decom-thm-tA} with Lemma \ref{lr-aA}  and  Proposition \ref{tA-aA-left}, 
we have the following corollary which will be used in the next section. 

\begin{cor}\label{decom-cor}
For $\la \in \vL^+$ and $\mu \in \vL^+_0$, 
the following holds. 
\begin{enumerate}
\item
If $(\mu,0)\in \Om_0$ and $\a(\la)=\a(\mu)$ then we have 
\[ [\wt{Z}^{(\la,0)}:\wt{L}^{(\mu,0)}]_{\tA}=[Z^{(\la,0)}:L^{(\mu,0)}]_{\aA}
	=[\,\sh Z^{(\la,0)}:\,\sh L^{(\mu,0)}]_{\aA}=[\,\sh \wt{Z}^{(\la,0)}:\,\sh\wt{L}^{(\mu,0)}]_{\tA}.\] 
\item
If $(\mu,0)\in \Om_0$ and $\a(\la)\not=\a(\mu)$ then we have 
\[[\wt{Z}^{(\la,1)}:\wt{L}^{(\mu,0)}]_{\tA} =[Z^{(\la,1)}:L^{(\mu,0)}]_{\aA}
	=[\,\sh Z^{(\la,1)}:\,\sh L^{(\mu,0)}]_{\aA}=[\,\sh \wt{Z}^{(\la,1)}:\,\sh \wt{L}^{(\mu,0)}]_{\tA} .\]
\item
If $(\mu,1)\in \Om_0$ then we have 
\[ [\wt{Z}^{(\la,1)}:\wt{L}^{(\mu,1)}]_{\tA}= [Z^{(\la,1)}:L^{(\mu,1)}]_{\aA}
	=[\,\sh Z^{(\la,1)}:\,\sh L^{(\mu,1)}]_{\aA}= [\,\sh\wt{Z}^{(\la,1)}:\,\sh\wt{L}^{(\mu,1)}]_{\tA}.\]
\end{enumerate}
\end{cor}

\exam 
\label{ex5}
We give the parabolic type subalgebra for each algebra in Example \ref{ex} under the setting of Example \ref{ex2} and Example \ref{ex3}. 

(\roi) (Matrix algebra) 
We have $\wt{\ZC}^\a=\{E_{ij}\,|\,b\leq i,j \leq n\} \cup \{E_{ij}\,|\, 1 \leq i \leq b-1, \,1\leq j \leq n\}$, thus 
{\footnotesize 
\[\tA={\footnotesize \left( \begin{array}{ccc|ccc}
	* & \cdots &* & * &\cdots & * \\
	\vdots& \ddots& \vdots& \vdots & \ddots & \vdots\\
	* & \cdots &* & * &\cdots  &* \\\hline
	 &  & & * & \cdots &* \\
		&\scalebox{2}{0}& &\vdots&\ddots&\vdots\\
	 && & * & \cdots &* \\
	\end{array} \right),}\qquad 
(\tA)^\ast={\footnotesize \left( \begin{array}{ccc|ccc}
	* & \cdots &* &  & & \\
	\vdots& \ddots& \vdots&&\scalebox{2}{0}&\\
	* & \cdots &* &  & & \\\hline
	* & \cdots &* & * & \cdots &* \\
	\vdots& \ddots&\vdots &\vdots&\ddots&\vdots\\
	 *& \cdots &* & * & \cdots &* \\
	\end{array} \right).}  
\]}

(\roii) (Path algebra) 
We have 
\begin{align*}
\wt{\ZC}^\a= \left\{\a_{12}\a_{21}, \quad
\begin{matrix} e_1, & \hspace{-2mm} \a_{12},\\ \a_{21}, & \hspace{-2mm} \a_{21}\a_{12}, \end{matrix} \quad
\begin{matrix} e_2, & \hspace{-2mm} \a_{23},\\ \a_{32}, & \hspace{-2mm} \a_{32}\a_{23}, \end{matrix} \quad
\begin{matrix} e_3, & \hspace{-2mm} \\  \a_{43}, & \hspace{-2mm}  \a_{45}\a_{54}, \end{matrix} \quad
\begin{matrix} e_4, & \hspace{-2mm} \a_{45},\\ \a_{54}, & \hspace{-2mm} \a_{54}\a_{45}, \end{matrix} \quad
e_5\right\},\\ \\
(\wt{\ZC}^\a)^\ast= \left\{\a_{12}\a_{21}, \quad
\begin{matrix} e_1, & \hspace{-2mm} \a_{12},\\ \a_{21}, & \hspace{-2mm} \a_{21}\a_{12}, \end{matrix} \quad
\begin{matrix} e_2, & \hspace{-2mm} \a_{23},\\ \a_{32}, & \hspace{-2mm} \a_{32}\a_{23}, \end{matrix} \quad
\begin{matrix} e_3, & \hspace{-2mm} \a_{34},\\   & \hspace{-2mm}  \a_{45}\a_{54}, \end{matrix} \quad
\begin{matrix} e_4, & \hspace{-2mm} \a_{45},\\ \a_{54}, & \hspace{-2mm} \a_{54}\a_{45}, \end{matrix} \quad
e_5\right\}.
\end{align*}

Moreover, we see that 
\[\tA \cong (R\wt{Q}^\a/\wt{\CI}^\a), \qquad (\tA)^\ast \cong (R \wt{Q}^{\a\ast}/\wt{\CI}^{\a\ast})\]
where \\
\begin{picture}(100,30)
\put(10,8){$\wt{Q}^\a=\Big(1 \qquad \,\,2 \qquad \,\,3 \qquad \,\,4 \qquad\,\,5\Big),$}
\put(56,20){$ \underrightarrow{\,\,\a_{12}\,\,}$}
\put(56,0){$ \overleftarrow{\,\,\a_{21}\,\,}$}
\put(90,20){$ \underrightarrow{\,\,\a_{23}\,\,}$}
\put(90,0){$ \overleftarrow{\,\,\a_{32}\,\,}$}
\put(124,0){$ \overleftarrow{\,\,\a_{43}\,\,}$}
\put(158,20){$ \underrightarrow{\,\,\a_{45}\,\,}$}
\put(158,0){$ \overleftarrow{\,\,\a_{54}\,\,}$}
\put(220,8){$\wt{Q}^{\a\ast}=\Big(1 \qquad \,\,2 \qquad \,\,3 \qquad \,\,4 \qquad\,\,5\Big),$}
\put(270,20){$ \underrightarrow{\,\,\a_{12}\,\,}$}
\put(270,0){$ \overleftarrow{\,\,\a_{21}\,\,}$}
\put(304,20){$ \underrightarrow{\,\,\a_{23}\,\,}$}
\put(304,0){$ \overleftarrow{\,\,\a_{32}\,\,}$}
\put(338,20){$ \underrightarrow{\,\,\a_{34}\,\,}$}
\put(372,20){$ \underrightarrow{\,\,\a_{45}\,\,}$}
\put(372,0){$ \overleftarrow{\,\,\a_{54}\,\,}$}
\end{picture}
\begin{align*}
\wt{\CI}^\a=\lan 
	\a_{12}\a_{23},\, \a_{54}\a_{43},\, \a_{43}\a_{32},\, \a_{32}\a_{21},\,
	\a_{21}\a_{12}-\a_{23}\a_{32}, \, 
	\a_{45}\a_{54}\a_{45}, \, \a_{54}\a_{45}\a_{54}\ran_{\text{ideal}}, \\
\wt{\CI}^{\a\ast}=\lan 
	\a_{12}\a_{23},\, \a_{23}\a_{34},\, \a_{34}\a_{45},\, \a_{32}\a_{21},\,
	\a_{21}\a_{12}-\a_{23}\a_{32}, \, 
	\a_{45}\a_{54}\a_{45}, \, \a_{54}\a_{45}\a_{54}\ran_{\text{ideal}}.
\end{align*}

(\roiii) (Cyclotomic $q$-Schur algebra) 
We follow the notation in Example \ref{ex3}, (\roiii). 
We denote by $\tSp$ the parabolic type subalgebra of $\Sc$ with respect to the map $\a_\Bp$. 
In this paper, we have proved  Theorem \ref{decom-thm-tA} by 
investigating the relation between $\tSp$ and $\Sp$ combined with Theorem \ref{decom-thm}. 
However,  
the Levi type subalgebra $\Sp$ does not appear in \cite{SW} 
(note that the symbol $\Sp$ in \cite{SW} denotes the parabolic type subalgebra $\tSp$ in our paper).
In \cite[Theorem 3.13]{SW}, 
we have shown that 
\[[W^{\Bla}:L^{\Bmu}]_\Sc=[\wt{Z}^{(\Bla,0)}:\wt{L}^{(\Bmu,0)}]_{\tSp}=[\ol{Z}^{\Bla}:\ol{L}^{\Bmu}]_{\oSp}\]
for $\Bla,\Bmu \in \vL^+$ such that $\a_\Bp(\Bla)=\a_\Bp(\Bmu)$ 
by studying the induced modules $\wt{Z}^{(\la,0)} \otimes_{\tSp}\Sc$ and $\wt{L}^{(\Bmu,0)}\otimes_{\tSp}\Sc$ 
instead of the relation between $\tSp$ and $\Sp$. 
We emphasize that the relation between $\tSp$ and $\Sp$ gives more information than the discussions in \cite{SW} 
(e.g. Theorem \ref{decom-thm-tA} (\roii),(\roiii), Corollary \ref{decom-cor} or the results of the following section).

\section{Blocks of $\A,\aA,\tA$ and $\oA$}
Throughout this section, we assume that $R$ is a field. 
For cellular algebras $\A$, $\aA$ and $\oA$, 
the general theory enables us to 
classify blocks of them by using right (or left) standard modules. 
(See Section \ref{cellular}.) 
However, such a classification of blocks for standardly based algebras is not known. 
Nevertheless in our setting, we can classify blocks of $\tA$ by using \lq\lq \,right" standard modules of $\tA$ 
thanks to the relations  with $\aA$ discussed in Section \ref{tA}. 
In this section, we shall study the relations for blocks among $\A$, $\tA$ and $\oA$. 

\para
Let $\eqA$ be the equivalent relation on $\vL^+$ determined by the notion of \lq \lq\,cell linked" 
with respect to (right) standard modules $W^\la$ ($\la\in \vL^+$) of $\A$ 
as in \ref{cell-linked}.
Similary, we denote by $\eqaA$ (resp. $\eqoA$) 
an equivalent relation on $\Om$ (resp. $\ol{\vL}^+$) with respect to (right) standard modules 
$Z^{(\la,\ve)}$ ($(\la,\ve)\in \Om$) of $\aA$  
(resp. $\ol{Z}^\la$ ($\la \in \ol{\vL}^+$) of $\oA$). 
Let $\vL^+/_\sim$ $($resp. $\Om/_\sim$, $\ol{\vL}^+/_\sim$$)$ 
be the set of equivalence classes on $\vL^+$ $($resp. $\Om$, $\ol{\vL}^+$$)$ with respect to the relation $\eqA$ $($resp. $\eqaA$, $\eqoA$$)$. 
The following proposition holds by \cite{GL96} (cf. Section \ref{cellular}).

\begin{prop}[{\cite{GL96}}]\ 
\begin{enumerate}
\item
For $\la,\mu\in \vL^+$, $\la \eqA \mu$ if and only if $W^\la$ and $W^\mu$ belong to the same block of $\A$. 
\item
For $(\la,\ve),(\mu,\ve')\in \Om$, $(\la,\ve) \eqaA (\mu,\ve')$ if and only if $Z^{(\la,\ve)}$ and $Z^{(\mu,\ve')}$ belong to the same block of $\aA$. 
\item
For $\la,\mu\in \ol{\vL}^+$, $\la \eqoA \mu$ if and only if $\ol{Z}^\la$ and $\ol{Z}^\mu$ belong to the same block of $\oA$. 
\end{enumerate}

Moreover, 
there is a one-to-one correspondence between $\vL^+/_\sim $ 
$($resp. $\Om/_\sim$, $\ol{\vL}^+/_\sim$$)$ 
and the set of blocks of $\A$ $($resp. $\aA$, $\oA$$)$.
\end{prop}

\para
From now on, 
we consider the classification of blocks of $\tA$ 
by using \lq\lq\,right" standard modules of $\tA$. 
The idea of the classification of blocks of $\tA$ is similar to the case of cellular algebras, but we need to use Corollary \ref{decom-cor} 
which does not hold in the general setting of standardly based algebras. 

\para
For $(\la,\ve) \in \Om_0$, we denote by $\ol{x}$ \big(resp. $\sh \ol{x}$\big)  
the image of $x \in \wt{Z}^{(\la,\ve)}$ \big(resp. $\sh x \in \,\sh \wt{Z}^{(\la,\ve)})$\big) 
under the natural surjection $\wt{Z}^{(\la,\ve)}\ra \wt{L}^{(\la,\ve)}$ \big(resp. $\sh \wt{Z}^{(\la,\ve)} \ra \,\sh \wt{L}^{(\la,\ve)}$\big). 
We define a bilinear form $\ol{\b}:\,\sh \wt{L}^{(\la,\ve)}\times \wt{L}^{(\la,\ve)}\ra R$ by 
$\ol{\b}(\,\sh \ol{x} ,\ol{y})=\b(\,\sh x,y)$, where $\b:\,\sh \wt{Z}^{(\la,\ve)} \times \wt{Z}^{(\la,\ve)}\ra R$ 
is the bilinear form defined in \ref{tA-general}. 
It is clear that $\ol{\b}$ is non-degenerate.

For a left $\tA$-module $\sh M$, 
we regard $\Hom_R (\,\sh M,R)$ as a right $\tA$-module by the standard way. 
For $(\la,\ve)\in \Om_0$, 
we define an $R$-module homomorphism $\Phi: \wt{L}^{(\la,\ve)} \ra \Hom_R(\,\sh \wt{L}^{(\la,\ve)},R)$ by 
$\Phi(\ol{x})=\ol{\beta}(-,\ol{x})$. 
Similarly, we  define an $R$-module homomorphism  $\sh \Phi: \,\sh \wt{L}^{(\la,\ve)} \ra \Hom_R(\wt{L}^{(\la,\ve)},R)$ 
by $\sh \Phi(\,\sh\ol{x})=\ol{\b}(\,\sh\ol{x},-)$. 
Then we have the following lemma. 

\begin{lem}\label{iso-lem}
For $(\la,\ve)\in \Om_0$, the following holds. 
\begin{enumerate}
\item
$\Phi: \wt{L}^{(\la,\ve)} \ra \Hom_R(\,\sh \wt{L}^{(\la,\ve)},R)$ is an isomorphism of right $\tA$-modules.
\item
$\sh \Phi: \,\sh \wt{L}^{(\la,\ve)} \ra \Hom_R(\wt{L}^{(\la,\ve)},R)$ is an isomorphism of left $\tA$-modules.
\end{enumerate}
In particular, we have $\dim_R \wt{L}^{(\la,\ve)}= \dim_R \,\sh \wt{L}^{(\la,\ve)}$. 
\end{lem}

\begin{proof}
One sees that $\Phi$ is a homomorphism of right $\tA$-modules 
by the associativity of the bilinear form $\b$ (\cite[Lemma 1.2.6]{DR98}). 
Since $\ol{\b}$ is  non-degenerate, $\Phi$ is injective, 
thus is a non-zero map. 
Since $\Hom_R(\,\sh \wt{L}^{(\la,\ve)},R)$ is a simple right $\tA$-module,  
we see that $\Phi$ is an isomorphism. 
The proof for $\sh \Phi$ is similar. 
\end{proof}

For $(\la,\ve)\in \Om_0$, let $\wt{P}^{(\la,\ve)}$ be a projective cover of $\wt{L}^{(\la,\ve)}$. 
We consider the multiplicity of $\wt{L}^{(\mu,\ve')}$ \big($(\mu,\ve')\in \Om_0$\big) in $\wt{P}^{(\la,\ve)}$,  
and denote it by $[\wt{P}^{(\la,\ve)}:\wt{L}^{(\mu,\ve')}]_{\tA}$.
Then we have the following proposition. 

\begin{prop}\label{proj-decom}
For $(\la,\ve_1),(\mu,\ve_2)\in \Om_0$, we have 
\begin{align}
 [\wt{P}^{(\la,\ve_1)}:\wt{L}^{(\mu,\ve_2)}]_{\tA}
=\sum_{(\nu,\ve) \in \Om}[\,\sh\wt{Z}^{(\nu,\ve)} :\,\sh\wt{L}^{(\la,\ve_1)}]_{\tA} \cdot [\wt{Z}^{(\nu,\ve)}:\wt{L}^{(\mu,\ve_2)}]_{\tA}.
\label{mul-P}
\end{align}
\end{prop}

\begin{proof}
Let $\vL^+=\{\la_1,\la_2,\cdots ,\la_k\}$ be such that $i<j$ if $\la_i > \la_j$, 
and set $\wt{Z}^{(\la,\ve)}=0$ if $(\la,\ve)\not\in \Om$.  
By the general theory of standardly based algebras (\cite{DR98}), 
$\wt{P}^{(\la,\ve_1)}$ has a filtration 
\[\wt{P}^{(\la,\ve_1)}=M_{(k,1)}\supseteqq M_{(k,0)} \supseteqq M_{(k-1,1)}\supseteqq \cdots \supseteqq M_{(1,1)}\supseteqq M_{(1,0)} \supseteqq 0\]
such that $M_{(i,\ve)}/M_{(i-\ve' ,\ve')}=\wt{P}^{(\la,\ve_1)}\ot_{\tA}\,\hspace{-2mm}\sh \wt{Z}^{(\la_i,\ve)} \ot_R \wt{Z}^{(\la_i,\ve)}$, 
where $\ve+\ve'=1$. 
This implies that 
\[[\wt{P}^{(\la,\ve_1)}:\wt{L}^{(\mu,\ve_2)}]_{\tA}
=\sum_{(\nu,\ve)\in \Om} \dim_R \big(\wt{P}^{(\la,\ve_1)}\ot_{\tA}\,\hspace{-2mm} \sh \wt{Z}^{(\nu,\ve)}\big)
	 \cdot [\wt{Z}^{(\nu,\ve)}:\wt{L}^{(\mu,\ve_2)}]_{\tA}.\]
Thus it is enough to show that 
$\dim_R \big( \wt{P}^{(\la,\ve_1)}\ot_{\tA}\,\hspace{-2mm} \sh \wt{Z}^{(\nu,\ve)}\big)=[\,\sh \wt{Z}^{(\nu,\ve)}:\,\sh\wt{L}^{(\la,\ve_1)}]_{\tA}$. 
We have 
\[\wt{P}^{(\la,\ve_1)}\ot_{\tA}\,\hspace{-2mm} \sh \wt{Z}^{(\nu,\ve)}
\cong \Hom_R(\wt{P}^{(\la,\ve_1)}\ot_{\tA}\,\hspace{-2mm} \sh \wt{Z}^{(\nu,\ve)},R)
\cong \Hom_{\tA}\big(\wt{P}^{(\la,\ve_1)}, \Hom_R(\,\sh \wt{Z}^{(\nu,\ve)},R)\big)\]
as $R$-modules by the universal adjointness of property of $\Hom$ and $\ot$. 
Thus we have 
\begin{align*}
\dim_R \big( \wt{P}^{(\la,\ve_1)}\ot_{\tA}\,\hspace{-2mm} \sh \wt{Z}^{(\nu,\ve)}\big)
&= \dim_R\big(\Hom_{\tA}\big(\wt{P}^{(\la,\ve_1)}, \Hom_R(\,\sh \wt{Z}^{(\nu,\ve)},R)\big)\big)\\
&=\big[\Hom_R(\,\sh \wt{Z}^{(\nu,\ve)},R) : \wt{L}^{(\la,\ve_1)}\big]_{\tA} \\
&=\big[\Hom_R(\,\sh \wt{Z}^{(\nu,\ve)},R) : \Hom_R(\,\sh\wt{L}^{(\la,\ve_1)},R)\big]_{\tA}, 
\end{align*}
where the second equality follows from the general property (e.g. \cite[Corollary A.16]{M-book}), 
and the third equality follows from Lemma \ref{iso-lem}. 
Since $R$ is a field, we have 
\[\big[\Hom_R(\,\sh \wt{Z}^{(\nu,\ve)},R) : \Hom_R(\,\sh\wt{L}^{(\la,\ve_1)},R)\big]_{\tA}=[\,\sh \wt{Z}^{(\nu,\ve)} :\,\sh \wt{L}^{(\la,\ve_1)}]_{\tA}.\] 
Thus the proposition is proved. 
\end{proof}

For $(\la,\ve)\in \Om_0$, let $\sh\wt{P}^{(\la,\ve)}$ be the projective cover of $\sh\wt{L}^{(\la,\ve)}$. 
Then we have 
\begin{align}
 [\,\sh\wt{P}^{(\la,\ve_1)}:\,\sh\wt{L}^{(\mu,\ve_2)}]_{\tA}
=\sum_{(\nu,\ve) \in \Om}[\,\wt{Z}^{(\nu,\ve)} :\,\wt{L}^{(\la,\ve_1)}]_{\tA} \cdot [\,\sh\wt{Z}^{(\nu,\ve)}:\,\sh\wt{L}^{(\mu,\ve_2)}]_{\tA}.
\label{mul-shP}
\end{align}
in a similar way as in Proposition \ref{proj-decom}. 
By comparing the right-hand side of (\ref{mul-P}) with the right-hand side of (\ref{mul-shP}), 
we have the following corollary.

\begin{cor}\label{proj-cor}
For $(\la,\ve_1),(\mu,\ve_2) \in \Om_0$, we have 
\[[\wt{P}^{(\la,\ve_1)}:\wt{L}^{(\mu,\ve_2)}]_{\tA}=[\,\sh\wt{P}^{(\mu,\ve_2)}:\,\sh\wt{L}^{(\la,\ve_1)}]_{\tA}.\]
\end{cor}

\remark
Lemma \ref{iso-lem}, Proposition \ref{proj-decom} and Corollary \ref{proj-cor} hold for any standardly based algebra 
in the general setting.
\para
For $(\la,\ve), (\mu,\ve') \in \Om$, we say that $(\la,\ve)$ and $(\mu,\ve')$ are linked, and 
denote by $(\la,\ve) \eqtA (\mu,\ve')$  
if there exists a sequence $(\la,\ve)=(\la_1,\ve_1),(\la_2,\ve_2),\cdots ,(\la_k,\ve_k)=(\mu,\ve')$  
such that $\wt{Z}^{(\la_i,\ve_i)}$ and $\wt{Z}^{(\la_{i+1},\ve_{i+1})}$ have a common composition factor for $i=1,\cdots,k-1$. 
Then the relation $\eqtA$ turns out to be an equivalence relation on $\Om$. 
We have the following proposition. 

\begin{prop}\ \label{block-prop} For $(\la,\ve),(\mu,\ve')\in \Om$, 
the following holds. 
\begin{enumerate}
\item 
All of the composition factors of $\wt{Z}^{(\la,\ve)}$ lie in the same block. 
\item
$(\la,\ve) \eqtA (\mu,\ve')$ if and only if $\wt{Z}^{(\la,\ve)}$ and $\wt{Z}^{(\mu,\ve')}$ belong to the same block. 
\end{enumerate}
\end{prop}

\begin{proof}
(\roi) 
Let $\wt{L}^{(\mu_1,\ve_1)}$ and $\wt{L}^{(\mu_2,\ve_2)}$ be composition factors of $\wt{Z}^{(\la,\ve)}$. 
Then we have 
\[[\wt{Z}^{(\la,\ve)} :\wt{L}^{(\mu_i,\ve_i)}]_{\tA}\not=0 \text{ for }i=1,2.\] 

In the case where $\ve=0$, we have $\ve_1=\ve_2=0$ and $\a(\la)=\a(\mu_1)=\a(\mu_2)$ 
by Proposition \ref{decom-prop} (\roii), (\roiv) and Proposition \ref{tA-aA}.
Thus by Corollary \ref{decom-cor} (\roi), we have 
\[[\,\sh\wt{Z}^{(\la,0)} :\,\sh\wt{L}^{(\mu_1,0)}]_{\tA}=[\wt{Z}^{(\la,0)} :\wt{L}^{(\mu_1,0)}]_{\tA}\not=0.\] 
Combining this with Proposition \ref{proj-decom}, we have 
\[[\wt{P}^{(\mu_1,0)}:\wt{L}^{(\mu_2,0)}]_{\tA}
\geq [\,\sh\wt{Z}^{(\la,0)}:\,\sh\wt{L}^{(\mu_1,0)}]_{\tA}\cdot [\wt{Z}^{(\la,0)}:\wt{L}^{(\mu_2,0)}]_{\tA} \not=0.\]
Thus 
$\wt{P}^{(\mu_1,0)}$ and $\wt{P}^{(\mu_2,0)}$ have a common composition factor $\wt{L}^{(\mu_2,0)}$. 
This implies that $\wt{P}^{(\mu_1,0)}$ and $\wt{P}^{(\mu_2,0)}$ belong to the same block  
since $\wt{P}^{(\mu_1,0)}$ and $\wt{P}^{(\mu_2,0)}$ are principal indecomposable modules of $\tA$. 
It follows that $\wt{L}^{(\mu_1,0)}$ and $\wt{L}^{(\mu_2,0)}$ lie in the same block. 

Next we consider the case where $\ve=1$ 
and one of the $\ve_1, \ve_2$ is equal to $1$. 
We may assume that $\ve_1=1$. 
By Corollary \ref{decom-cor}\nobreak\,\nobreak (\roiii), we have 
\[[\,\sh\wt{Z}^{(\la,1)} :\,\sh\wt{L}^{(\mu_1,1)}]_{\tA}=[\wt{Z}^{(\la,1)} :\wt{L}^{(\mu_1,1)}]_{\tA}\not=0.\]
Combining this with Proposition \ref{proj-decom}, we have 
\[[\wt{P}^{(\mu_1,1)}:\wt{L}^{(\mu_2,\ve_2)}]_{\tA}
\geq [\,\sh\wt{Z}^{(\la,1)}:\,\sh\wt{L}^{(\mu_1,1)}]_{\tA}\cdot [\wt{Z}^{(\la,1)}:\wt{L}^{(\mu_2,\ve_2)}]_{\tA} \not=0.\] 
This implies that $\wt{L}^{(\mu_1,1)}$ and $\wt{L}^{(\mu_2,\ve_2)}$ lie in the same block 
by a similar reason as in the case of $\ve=0$. 

Finally, we consider the case where $\ve=1$ and $\ve_1=\ve_2=0$. 
If $\a(\la)\not=\a(\mu_i)$ for some $i$, say $i=1$,  then we have 
\[[\,\sh\wt{Z}^{(\la,1)} :\,\sh\wt{L}^{(\mu_1,0)}]_{\tA}=[\wt{Z}^{(\la,1)} :\wt{L}^{(\mu_1,0)}]_{\tA}\not=0 \]
by Corollary \ref{decom-cor} (\roii). 
Combining this with Proposition \ref{proj-decom}, we have 
\[[\wt{P}^{(\mu_1,0)}:\wt{L}^{(\mu_2,0)}]_{\tA}
\geq [\,\sh\wt{Z}^{(\la,1)}:\,\sh\wt{L}^{(\mu_1,0)}]_{\tA}\cdot [\wt{Z}^{(\la,1)}:\wt{L}^{(\mu_2,0)}]_{\tA} \not=0.\]
This implies that $\wt{L}^{(\mu_1,0)}$ and $\wt{L}^{(\mu_2,0)}$ lie in the same block 
by a similar reason as in the case of $\ve=0$. 
If $\a(\la)=\a(\mu_1)=\a(\mu_2)$ then we have 
\begin{align*}
[\wt{Z}^{(\la,0)} : \wt{L}^{(\mu_i,0)}]_{\tA}&=[Z^{(\la,0)}: L^{(\mu_i,0)}]_{\aA}\\
	&=[Z^{(\la,0)}\oplus Z^{(\la,1)} : L^{(\mu_i,0)}]_{\aA}\\
	&=[\wt{Z}^{(\la,1)}:\wt{L}^{(\mu_i,0)}]_{\tA}\\
	& \not=0 \qquad(\text{ for }i=1,2) 
\end{align*}
by Proposition \ref{decom-prop} (\roiii) and Proposition \ref{tA-aA}. 
This implies that $\wt{L}^{(\mu_1,0)}$ and $\wt{L}^{(\mu_2,0)}$ lie in the same block 
by the the result of $\ve=0$. 

(\roii) 
By (\roi), it is clear that if $(\la,\ve)\eqtA (\mu,\ve')$ then $\wt{Z}^{(\la,\ve)}$ and $\wt{Z}^{(\mu,\ve')}$ belong to the same block. 
We show the converse.
Suppose that $\wt{Z}^{(\la,\ve)}$ and $\wt{Z}^{(\mu,\ve')}$ belong to the same block. 
Then a composition factor $\wt{L}^{(\la,\ve)}$ of $\wt{Z}^{(\la,\ve)}$ and 
	a composition factor $\wt{L}^{(\mu,\ve')}$ of $\wt{Z}^{(\mu,\ve')}$ lie in the same block. 
This implies that $\wt{P}^{(\la,\ve)}$ and $\wt{P}^{(\mu,\ve')}$ belong to the same block. 
Thus there exist sequences $\la=\la_1,\la_2,\cdots,\la_k=\mu$, $\ve=\ve_1, \ve_2,\cdots, \ve_k=\ve'$ 
such that 
$\wt{P}^{(\la_i,\ve_i)}$ and $\wt{P}^{(\la_{i+1},\ve_{i+1})}$ have a common composition factor $\wt{L}^{(\nu_i,\ve_i'')}$. 
By Proposition \ref{proj-decom}, 
one sees that there exists $(\t_1,\ve_1)\in \Om$ such that 
\[ [\,\sh \wt{Z}^{(\t_1,\ve_1)}:\,\sh \wt{L}^{(\la_i,\ve_i)}]_{\tA}\not=0,\quad  [\wt{Z}^{(\t_1,\ve_1)}:\wt{L}^{(\nu_i,\ve_i'')}]_{\tA}\not=0,\] 
and there exists $(\t_2,\ve_2)\in \Om$ such that 
\[ [\,\sh \wt{Z}^{(\t_2,\ve_2)}:\,\sh \wt{L}^{(\la_{i+1},\ve_{i+1})}]_{\tA}\not=0,\quad  [\wt{Z}^{(\t_2,\ve_2)}:\wt{L}^{(\nu_i,\ve_i'')}]_{\tA}\not=0.\] 
By Proposition \ref{decom-prop} (\roiii), Lemma \ref{lr-aA} and  Proposition \ref{tA-aA-left}, 
we have  $\a(\t_1)\not=\a(\la_i)$ if $\ve_1=1$ and  $\ve=0$, 
and we have $\a(\t_2)\not=\a(\la_{i+1})$ if $\ve_2=1$ and $\ve'=0$. 
By Corollary \ref{decom-cor}, we have 
\begin{align*}
&[\wt{Z}^{(\t_1,\ve_1)}:\wt{L}^{(\la_i,\ve_i)}]_{\tA}=[\,\sh \wt{Z}^{(\t_1,\ve_1)}:\,\sh \wt{L}^{(\la_i,\ve_i)}]_{\tA}\not=0, \\
&[\wt{Z}^{(\t_2,\ve_2)}:\wt{L}^{(\la_{i+1},\ve_{i+1})}]_{\tA}=[\,\sh \wt{Z}^{(\t_2,\ve_2)}:\,\sh \wt{L}^{(\la_{i+1},\ve_{i+1})}]_{\tA}\not=0.
\end{align*}
Thus we have 
\[(\la_i,\ve_i)\eqtA (\t_1,\ve_1) \eqtA (\nu_i,\ve_i'') \eqtA (\t_2,\ve_2)  \eqtA (\la_{i+1},\ve_{i+1}) \,\,\text{ for any }i=1,\cdots,k-1, \]
and we have $(\la,\ve) \eqtA (\mu,\ve')$. 
The proposition is proved. 
\end{proof}
For $\la \in \vL^+$, if $(\la,0), (\la,1) \in \Om$ then $\wt{Z}^{(\la,0)}$ 
is an $\tA$-submodule of $\wt{Z}^{(\la,1)}$ 
(see Proposition \ref{tA-A}). 
Thus we have $(\la,0) \eqtA (\la,1)$ by Proposition \ref{block-prop}. 
Therefore, we can define the equivalence relation $\eqtAt$ on $\vL^+$ by $\la \eqtAt \mu $ 
if and only if $\wt{Z}^{(\la,\ve)}$ $(\ve=0,1)$ and $\wt{Z}^{(\mu,\ve')}$ $(\ve'=0,1)$ belong to same block of $\tA$. 
Let $\vL^+/_\eqtAt$ 
be the set of equivalence classes on $\vL^+$  with respect to the relation $\eqtAt$. 
The following result gives a  classification of blocks of $\tA$. 
\begin{cor}\ 
There exists a bijective correspondence between $\vL^+/_\eqtAt $ 
and the set of blocks of $\tA$.
\end{cor}

Moreover, we have  the following relations among blocks of $\A$,  $\tA$ and  $\oA$. 
\begin{thm}\ \label{relation-block}
\begin{enumerate}
\item
For $\la,\mu \in \vL^+$, $\la \eqA \mu$ if and only if $\la \eqtAt \mu$. 
In particular, $\vL^+/_\sim$ coincides with $\vL^+/_\eqtAt$. 
\item
For $\la,\mu \in \ol{\vL}^+$, $\la \eqA \mu $ and $\a(\la)=\a(\mu)$ if and only if $\la \eqoA \mu$. 
\end{enumerate}
\end{thm}
\begin{proof}
(\roi) follows from  Theorem \ref{decom-thm-tA}. 
(\roii) also follows from Theorem \ref{decom-thm-tA} 
by taking Proposition \ref{decom-prop} (\rov) into account. 
\end{proof}

\remarks\ \label{remark-block}

(\roi) 
We can also prove that all the composition factors of $\sh \wt{Z}^{(\la,\ve)}$ \big($(\la,\ve) \in \Om$\big) lie in the same block 
in a similar way as in the proof of Proposition \ref{block-prop}. 
Thus, for $(\la,\ve),(\mu,\ve') \in \Om$,  if $(\la,\ve) \eqtAl (\mu,\ve')$ 
then $\sh \wt{Z}^{(\la,\ve)}$ and $\sh \wt{Z}^{(\mu,\ve')}$ belong to the same block of $\tA$, 
where the relation $\eqtAl$ 
is defined in a similar way as in the case of right standard modules 
but by replacing the right modules by left modules. 
However one can not prove the converse 
since it may happen that  $[\,\sh \wt{Z}^{(\la,1)} : \,\sh \wt{L}^{(\mu,0)}]_{\tA}=0$ 
even if $[\wt{Z}^{(\la,1)} : \wt{L}^{(\mu,0)}]_{\tA}\not=0$. (Consider the case of $\a(\la)=\a(\mu)$.)

(\roii) 
We see easily  that  $\la \eqA \mu$ if $(\la,\ve)\eqaA (\mu,\ve')$. 
But we can not show $(\la,\ve)\eqaA (\mu,\ve')$ even if $\la \eqA \mu$. 
For example, it happens that $[Z^{(\la,0)}:L^{(\nu,0)}]_{\aA}\not=0, [Z^{(\la,0)}:L^{(\nu,1)}]_{\aA}=0, 
[Z^{(\mu,1)}:L^{(\nu,0)}]_{\aA}=0$ and $[Z^{(\mu,1)}:L^{(\nu,1)}]_{\aA}\not=0$ even if $[W^\la:L^\nu]_\A\not=0$ and $[W^\mu:L^\nu]_\A\not=0$. 
In this case, we have $\la \eqA \mu$, 
but we don't know whether $(\la,0)\eqtA (\mu,1)$ or not. 

\para
Let 
\begin{align}
\A=\bigoplus_{\vG\in \vL^+/_\sim}\CB_{\vG}
\label{block-decom-A}
\end{align}
be the block decomosition of $\A$ as (\ref{block-decom}), 
and $e_\vG$ be the block idempotent of $\A$ such that $e_\vG \A=\CB_\vG$ 
for $\vG \in \vL^+/_\sim$.  

Put $\wt{\CB}_{\vG}^\a=\CB_{\vG}\cap \tA$ for $\vG\in \vL^+/_\sim$. 
Then $\wt{\CB}_{\vG}^\a$ is a two-sided ideal of $\tA$.
We have $\tA=\bigoplus_{\vG}e_\vG \tA$ 
from the decomposition of the unit element of $\A$ into the block idempotents $1_\A=\sum_{\vG} e_\vG$. 
Thus we have $e_\vG \tA \subset \wt{\CB}_{\vG}^\a$. 
This implies that 
\begin{align}
\tA=\bigoplus_{\vG\in \vL^+/_\sim}\wt{\CB}_{\vG}^\a. 
\label{block-decom-tA}
\end{align}
In fact, this decomposition  gives the block decomposition of $\tA$ 
since the number of blocks of $\tA$ is equal to the number of blocks of $\A$ by Theorem \ref{relation-block} (\roi).   

Similarly, 
put $\CB_\vG^\a=\CB_\vG\cap \aA$ for $\vG\in \vL^+/_\sim$. 
Then $\CB_{\vG}^\a$ is a two-sided ideal of $\aA$, and we have 
\begin{align}
\aA=\bigoplus_{\vG\in \vL^+/_\sim}\CB_{\vG}^\a. 
\label{pre-block-decom-aA}
\end{align}
However, (\ref{pre-block-decom-aA}) does not give the block decomposition of $\aA$ in general 
(see Remark \ref{remark-block} (\roii)). 
Hence, $\CB_\vG^\a$ is a sum of some blocks of $\aA$.    

Let $\ol{\CB}_\vG^\a$ be the image of $\wt{\CB}_\vG^\a$ under the natural surjection $\tA \ra \oA$, 
which  coincides with the image of $\CB_\vG^\a$ under the natural surjection $\aA \ra \oA$. 
Thus we have 
\begin{align}
\ol{\CB}_\vG^\a = \wt{\CB}_\vG^\a \big/ \big(\wt{\CB}_\vG^\a \cap \tA(\wh{\Om})\big)
	=\CB_\vG^\a \big/ \big(\CB_\vG^\a \cap \aA(\wh{\Om})\big).
\label{B-bar}
\end{align} 

On the other hand, put $X^+=\{\a(\la)\,|\,\la \in \vL^+\}$. 
For $\vG \in \vL^+/_\sim$, $\eta  \in X^+$, 
put $\vG_\eta=\{\la\,|\,\la \in \vG \text{ such that }\a(\la)=\eta \}$. 
Then, by Theorem \ref{relation-block} (\roii), 
we have the block decomposition of $\oA$; 
\begin{align}
\oA=\bigoplus_{\vG\in \vL^+/_\sim\,} \bigoplus_{\,\eta\in X^+}\ol{\CB}_{\vG_\eta}^\a, 
\label{block-decom-oA}
\end{align}
where we put $\ol{\CB}_{\vG_\eta}^\a=0$ if  $\la \not\in \ol{\vL}^+$ for any $\la \in \vG_\eta$. 
By Theorem \ref{relation-block}, 
we see that, for $\la \in \ol{\vL}^+$,
$\ol{Z}^\la$ belongs to 
$\bigoplus_{\,\eta\in X^+}\ol{\CB}_{\vG_\eta}^\a$ 
if and only if 
$\wt{Z}^{(\la,0)}$ belongs to $\wt{\CB}_\vG^\a$. 
This implies that 
\begin{align}
\ol{\CB}_\vG^\a=\bigoplus_{\,\eta\in X^+}\ol{\CB}_{\vG_\eta}^\a \quad \text{ for }\vG \in \vL^+/_\sim.
\end{align}

Summing up the above arguments, 
we have the following commutative diagram for each $\vG\in \vL^+/_\sim$;
 
\hspace{5em}
\begin{picture}(30,60)
\put(0,40){$\CB_{\vG}^\a$ \scalebox{2}[1]{$\hookrightarrow $} $\wt{\CB}_{\vG}^\a $ \scalebox{2}[1]{$\hookrightarrow $} $\CB_\vG$}
\put(50,-1){$\displaystyle \ol{\CB}_{\vG}^\a = \bigoplus_{\eta\in X^+}\ol{\CB}_{\vG_\eta}^\a,$}
\put(5,8){\rotatebox{140}{\scalebox{4}[0.6]{$\twoheadleftarrow$}}}
\put(53,11){\rotatebox{90}{\scalebox{2.2}[1.2]{$\twoheadleftarrow$}}}
\end{picture}\vspace{2em}\\
where $\CB_{\vG}$ (resp. $\wt{\CB}_{\vG}^\a$, $\ol{\CB}_{\vG_\eta}^\a$) is a block of $\A$ (resp. $\tA$, $\oA$), 
but $\CB_{\vG}^\a$ may not be a block of $\aA$.  

\remark 
By Proposition \ref{basis-block}, $\CB_\vG$ is a cellular algebra with a cellular basis 
$\ZC_\vG=\{b_{\Fs\Ft}^\la \,|\,\Fs,\Ft \in \CT(\la) \text{ for some }\la \in \vG\}$ 
such that $b_{\Fs\Ft}^\la \equiv c_{\Fs\Ft}^\la \mod \A^{\vee \la}$ in $\A$. 
Then one may expect that $\CB_\vG^\a$ (resp. $\wt{\CB}_\vG^\a$) is a Levi type (resp. parabolic type) subalgebra of $\CB_\vG$ 
with respect to the map $\a$. 
But there exist some problems to overcome. 
\vspace{3mm}

(\roi)
Does it hold that $c_{\Fs\Ft}^{\la} e_\mu=c_{\Fs\Ft}^{\la}$ if and only if $b_{\Fs\Ft}^{\la} e_{\mu}=b_{\Fs\Ft}^{\la}$ 
for $\la \in \vG, \,\Ft\in \CT(\la), \,\mu \in \vL$ ? 
\vspace{3mm}\\
Note that $e_\vG=\sum_{\mu\in vL}e_\mu e_\vG$ gives a decomposition of $e_\vG$ into orthogonal idempotents 
since $e_\vG$ is an element of the center of $\A$, 
and that  $b_{\Fs\Ft}^{\la} e_{\mu}=b_{\Fs\Ft}^{\la}$ if and only if $b_{\Fs\Ft}^\la e_{\mu}e_\vG=b_{\Fs\Ft}^\la$ 
since $b_{\Fs\Ft}^\la e_\vG=b_{\Fs\Ft}^\la$. 
The \lq\lq if part" in (\roi) always holds 
since $b_{\Fs\Ft}^\la \equiv c_{\Fs\Ft}^\la \mod \A^{\vee \la}$ 
and $c_{\Fs\Ft}^\la e_\mu$ is equal to $c_{\Fs\Ft}^\la$ or zero. 
However, \lq\lq only if part" in (\roi) is not clear. 
\vspace{3mm}

(\roii) 
Does $\ZC_\vG\cap \CB_\vG^\a $ (resp. $\ZC_\vG \cap \wt{\CB}_\vG^\a$) give a cellular basis of $\CB_\vG^\a$ 
(resp. a standard basis of $\wt{\CB}_\vG^\a$) ? 
\vspace{3mm}\\
If (\roii) holds, then we have 
$\ZC_\vG\cap \CB_\vG^\a=\{b_{\Fs\Ft}^\la\,|\,\Fs,\Ft \in I(\la,\ve) \text{ for some } \la \in \vG,\,\ve=0,1\}$ and 
$\ZC_\vG\cap \CB_\vG^\a=\{b_{\Fs\Ft}^\la\,|\,(\Fs,\Ft) \in \wt{I}(\la,\ve)\times \wt{J}(\la,\ve) \text{ for some } \la \in \vG,\,\ve=0,1\}$. 


\end{document}